\documentclass[10pt,reqno]{amsart}

\RequirePackage[OT1]{fontenc}
\RequirePackage{amsthm,amsmath}
\RequirePackage[numbers]{natbib}
\RequirePackage[colorlinks,linkcolor=blue,citecolor=blue,urlcolor=blue]{hyperref}

\usepackage{amssymb}
\usepackage{anysize}
\usepackage{hyperref}
\usepackage{extarrows}
\usepackage{multirow}
\usepackage{natbib}
\usepackage{stmaryrd}
\usepackage{color}
\usepackage{amsmath}
\usepackage{enumitem}
\usepackage{changes}

\allowdisplaybreaks

\theoremstyle{plain} 
\newtheorem{thm}{Theorem}[section]
\newtheorem*{thm*}{Theorem}
\newtheorem{lem}[thm]{Lemma}
\newtheorem{cor}[thm]{Corollary}
\newtheorem*{cor*}{Corollary}
\newtheorem{pro}[thm]{Proposition}
\newtheorem*{pro*}{Proposition}
\newtheorem{defn}[thm]{Definition}
\newtheorem*{defn*}{Definition}

\theoremstyle{definition} 

\newtheorem*{example*}{Example}
\newtheorem{rem}[thm]{Remark}
\newtheorem*{rem*}{Remark}

\renewcommand{\Re}{\mathrm{Re}\,}
\renewcommand{\Im}{\mathrm{Im}\,}
\newcommand{\re}{\mathrm{Re}\,}
\newcommand{\im}{\mathrm{Im}\,}

\newcommand{\E}{{\mathbb E }}
\newcommand{\R}{{\mathbb R }}
\newcommand{\N}{{\mathbb N}}

\renewcommand{\P}{{\mathbb P}}
\newcommand{\C}{{\mathbb C}}

\newcommand{\ii}{\mathrm{i}}
\newcommand{\deq}{\mathrel{\mathop:}=}
\newcommand{\e}[1]{\mathrm{e}^{#1}}
\newcommand{\ntr}{\mathrm{tr}\,}
\newcommand{\dd}{\mathrm{d}}
\newcommand{\ie}{{\it i.e.\ }}
\newcommand{\eg}{{\it e.g.\ }}
\newcommand{\cf}{{\it cf.\ }}
\newcommand{\PP}{\Phi}
\newcommand{\dL}{\mathrm{d}_{\mathrm{L}}}
\newcommand{\wt}{\widetilde}

\newcommand{\bs}{\boldsymbol}

\newcommand{\la}{\langle}
\newcommand{\ra}{\rangle}

\renewcommand{\mathbf}[1]{\bs{#1}}
 
 \numberwithin{equation}{section} 
\numberwithin{thm}{section}

\marginsize{1in}{1in}{1in}{1in}   

\newcommand\blfootnote[1]{%
  \begingroup
  \renewcommand\thefootnote{}\footnote{#1}%
  \addtocounter{footnote}{-1}%
  \endgroup
}

\begin{document}
 
\begin{minipage}{0.85\textwidth}
 \vspace{2.2cm}
    \end{minipage}
 
 \begin{center}
 \large\bf
 Convergence Rate for Spectral Distribution of Addition of Random Matrices\blfootnote{
 \begin{flushleft}
 {\textsuperscript{$*$}Supported by ERC Advanced Grant RANMAT No.\ 338804}\\
 {\textsuperscript{$\dagger$}Partially supported by ERC Advanced Grant RANMAT No.\ 338804}\\
{\hspace{4pt}\textit{Keywords}:  Random matrices, strong local law, free convolution, convergence rate}\\
{\hspace{4pt}\textit{AMS Subject Classification (2010)}: 46L54, 60B20}\\
{\hspace{4pt}\textit{Date}: June 9, 2016 }\\
 \end{flushleft}}
 \end{center}

\vspace{1.5cm}
\begin{center}
 \begin{minipage}{0.3\textwidth}
\begin{center}
Zhigang Bao\textsuperscript{$*$} \\
\footnotesize {IST Austria}\\
{\it zhigang.bao@ist.ac.at}
\end{center}
\end{minipage}
\begin{minipage}{0.3\textwidth}
\begin{center}
L\'aszl\'o Erd{\H o}s\textsuperscript{$\dagger$}  \\
\footnotesize {IST Austria}\\
{\it lerdos@ist.ac.at}
\end{center}
\end{minipage}
\begin{minipage}{0.3\textwidth}
 \begin{center}
Kevin Schnelli\textsuperscript{$*$}\\
\footnotesize 
{IST Austria}\\
{\it kevin.schnelli@ist.ac.at}
\end{center}
\end{minipage}

\end{center}
\vspace{1.5cm}

\begin{center}
 \begin{minipage}
 {0.9\textwidth}
\small
\hspace{10pt}
 Let $A$ and $B$ be two $N$ by $N$ deterministic Hermitian matrices and let $U$ be an $N$ by $N$ Haar distributed unitary matrix. It is well known that the 
 spectral distribution of the sum $H=A+UBU^*$ converges weakly to the free additive convolution of the spectral distributions of $A$ and $B$, as $N$ tends to infinity. We establish the optimal convergence rate ${\frac{1}{N}}$ in the bulk of the spectrum.
\end{minipage}
\end{center}

\vspace{1cm}
 
\thispagestyle{headings}

\section{Introduction}
In the influential work~\cite{Voi91}, Voiculescu showed that two independent large Hermitian matrices are asymptotically free if one of them is conjugated by a Haar distributed unitary matrix. This observation identifies the law of 
the sum of two large Hermitian matrices  in a randomly chosen relative basis. 
 More specifically, if $A=A^{(N)}$ and $B=B^{(N)}$ are two sequences of deterministic $N$ by $N$ Hermitian matrices and $U$ is a Haar distributed unitary matrix, then the empirical eigenvalue distribution, $\mu_H$, of the random sum $H=A+UBU^*$ is asymptotically given by the free additive convolution, $\mu_A\boxplus\mu_B$, of the eigenvalue distributions of $A$ and $B$. A quantitative control of the closeness between $\mu_H$ and $\mu_A\boxplus\mu_B$, or the convergence rate of $\mu_H$, has been out of reach until very recently.  The first convergence rate $(\log N)^{-1/2}$ was obtained by Kargin in~\cite{Kargin2012} by using the Gromov-Milman concentration inequality for the Haar measure. Later, Kargin improved  in~\cite{Kargin} his result to $N^{-1/7}$ in the bulk of the spectrum by studying the Green function subordination property down to the scale $N^{-1/7}$. Recently, we  used in~\cite{BES15} a bootstrap argument to successively localize the Gromov-Milman inequality from larger to smaller scales, whereby we improved the convergence rate to $N^{-2/3}$.  
 
In the current paper, we establish the convergence rate $N^{-1+\gamma}$, for any given $\gamma>0$, in the bulk regime. Since the typical eigenvalue spacing in the bulk of the spectrum is $N^{-1}$, our result is optimal, up to the $N^{\gamma}$ factor. In our recent work~\cite{BES15b} on the local law of $H$ we showed that the Green function subordination property holds down to the optimal scale $N^{-1+\gamma}$; \cf Proposition~\ref{thm. entrywise estimate} below. In particular, the fluctuations of the matrix elements of the Green function $G(z)=(H-z)^{-1}$ were
shown to be of order $N^{-1/2+\gamma}$ for any fixed  $z$ in the upper half plane, $\im z>0$. 
 To get the optimal convergence rate, we need to show that the fluctuations of the normalized trace
 of the Green function, $\frac{1}{N}\mbox{Tr}\, G$, are at most of order $N^{-1+\gamma}$.
 Thus the main task is to establish the fluctuation averaging of the diagonal entries of the Green function, \ie that the fluctuations of the (weighted) average of the $G_{ii}$'s are typically as small as the square of the fluctuation of the $G_{ii}$'s; \cf~\eqref{fluctuation averaging}. 
 
 Alongside with the convergence rate of $\mu_H$ to $\mu_{A}\boxplus\mu_B$, the concentration rate of $\mu_H$ to its expectation $\E\mu_H$ is of interest. An order $N^{-1/2}$ estimate up to logarithmic corrections on the fluctuations of the distribution function was obtained by Chatterjee in~\cite{Chatterjee} by studying mixing times of random walks on the unitary group. Using the Gromov-Milman concentration inequality, Kargin removed the logarithmic corrections~\cite{Kargin2012}. More recently, a rate of order $N^{-2/3}$ in the $\mathrm{L}^1$-Wasserstein distance was obtained by E.\ Meckes and M.\ Meckes in~\cite{MeckesMeckes2013}. From our main result it follows that~$\mu_H$, when restricted to the bulk, has concentration rate $N^{-1}$.

 The fluctuation averaging mechanism is a key ingredient in
proving the optimal convergence rate of local laws for random matrices.
 It was first introduced in~\cite{EYYBer} and substantially
extended later in~\cite{EKY, EKYY13} to generalized Wigner matrices.  In all previous works, however, the proofs
heavily relied on  the independence (up to symmetry) of the matrix elements.
Our matrix $H=A+UBU^*$ lacks this independence 
since the columns of a Haar unitary matrix are dependent. This fact was already a major obstacle  in 
the proof of the optimal local law~\cite{BES15b}, where independence of columns was replaced with a specific 
{\it partial randomness
decomposition} of the Haar unitaries; see Section~\ref{s.PRD}. This decomposition, however, is not
directly compatible with taking matrix elements of the Green function, thus the fluctuation averaging mechanism
in the average $\frac{1}{N}\sum_{i=1}^N G_{ii}$ remains hidden. In fact, our proof does not attack this average directly,
we first prove fluctuation averaging for an auxiliary quantity~$Z_i$, a carefully chosen linear 
combination of $G_{ii}$ and $(UBU^*G)_{ii}$; see~\eqref{def of Z}.  In the quantity~$Z_i$ certain fluctuations of
order $N^{-1/2}$ cancel for an algebraic reason. When passing from~$Z_i$ to the original $G_{ii}$, we need
to introduce an additional specially chosen quantity~$\Upsilon$, see~\eqref{definition of Upsilon}, that averages the effect of the fluctuations of order $N^{-1/2}$. Only {\it a posteriori} we show that~$\Upsilon$ is in fact one order better than its naive size indicates. Identifying these somewhat  counter-intuitive quantities for the fluctuation averaging is one of the main novelties of the current work.
 
 Another key feature of the proof is that we do not directly compute high moments of the averages as it was customary in the previous proofs that led to involved expansions whose bookkeeping was  quite tedious. 
 Instead, we estimate the higher moments recursively, in terms of the lower moments, see Lemma~\ref{lem. 022301},
 whose proof relies on integration by parts for Gaussian variables. This 
 method to prove fluctuation averaging was recently introduced in~\cite{LS16} in the context of sparse
 Wigner matrices.  In the current setup, circumventing the high moment calculation is a very
 important  asset, due to the complexity of
 the partial randomness decomposition of the Haar measure and the numerous error terms involved in
 the necessary Gaussian approximation.

{\it Notation:}
 We use $C$ to denote strictly positive constants that do not depend on $N$. Their values may change from line to line. For $a,b\ge0$, we write $a\lesssim b$, $a\gtrsim b$ if there is $C\ge1$ such that $a\le Cb$, $a\geq C^{-1} b$ respectively. We denote for $z\in\C^+$ the real part by $E=\re z$ and the imaginary part by $\eta=\im z$.

We use bold font for vectors in $\mathbb{C}^N$, denote their components by $\mathbf{v}=(v_1,\ldots,v_N)\in \mathbb{C}^N$ and their Euclidean norm by $\|\mathbf{v}\|_2$. The canonical basis of $\C^N$ is denoted by $(\mathbf{e}_i)_{i=1}^N$. We denote by $M_N(\mathbb{C})$ the set of $N\times N$ matrices over $\mathbb{C}$. For $A\in M_N(\mathbb{C})$, we denote by $\|A\|$ its operator norm and by $\|A\|_2$ its Hilbert-Schmidt norm.  The matrix entries of $A$ are denoted by $A_{ij}=\mathbf{e}_i^*A\mathbf{e}_j$. We use $\ntr A$ to denote the normalized trace of $A$, \ie $\ntr A=\frac{1}{N}\sum_{i=1}^N A_{ii}$. 

Let $\mathbf{g}=(g_1,\ldots, g_N)$ be a real or complex Gaussian vector. We write $\mathbf{g}\sim \mathcal{N}_{\mathbb{R}}(0,\sigma^2I_N)$ if $g_1,\ldots, g_N$ are i.i.d. $N(0,\sigma^2)$ normal variables; and we write $\mathbf{g}\sim \mathcal{N}_{\mathbb{C}}(0,\sigma^2I_N)$ if $g_1,\ldots, g_N$ are i.i.d. $N_{\mathbb{C}}(0,\sigma^2)$ variables, where $g_i\sim N_{\mathbb{C}}(0,\sigma^2)$ means that $\Re g_i$ and $\Im g_i$ are independent $N(0,\frac{\sigma^2}{2})$ normal variables. Finally, we use double brackets to denote index sets, \ie for $n_1,n_2\in\mathbb{R}$, $\llbracket  n_1,n_2\rrbracket:=[n_1,n_2]\cap \mathbb{Z}$.

\section{Main results} \label{s.results}
\subsection{Free additive convolution} For the reader's convenience we recall from~\cite{BES15}  some basic notions and results for the free additive convolution.

Given a probability measure $\mu$ on $\mathbb{R}$,  its {\it Stieltjes transform}, $m_\mu$, on the complex upper half-plane $\mathbb{C}^+\deq\{z\in \mathbb{C}: \Im z>0\}$ is defined by 
\begin{align*}
m_\mu(z)\deq \int_{\mathbb{R}} \frac{{\rm d} \mu(x)}{x-z},\qquad\qquad z\in \mathbb{C}^+.
\end{align*}
Note that $m_{\mu}\,:\,\C^+\rightarrow \C^+$ is an analytic function such that
\begin{align}\label{le limit to be a probablity measure}
 \lim_{\eta\nearrow\infty} \ii \eta\, m_\mu(\ii\eta)=-1.
\end{align}
Conversely, if $m\,:\, \C^+\rightarrow \C^+$ is an analytic function such that $\lim_{\eta\nearrow\infty} \ii \eta\, m(\ii\eta)=-1$, then $m$ is the Stieltjes transform of a probability measure $\mu$. Let $F_\mu$ be the {\it negative reciprocal Stieltjes transform} of $\mu$,
\begin{align}\label{le F definition}
 F_{\mu}(z)\deq -\frac{1}{m_{\mu}(z)},\qquad \qquad z\in\C^+.
\end{align}
Observe that
 \begin{align}\label{le F behaviour at infinity}
\lim_{\eta\nearrow \infty}\frac{F_{\mu}(\ii\eta)}{\ii\eta}=1,
\end{align}
as follows from~\eqref{le limit to be a probablity measure}. Note, moreover, that $F_\mu$ is analytic on $\C^+$ with nonnegative imaginary part.

The {\it free additive convolution} is the symmetric binary operation on probability measures on $\R$ characterized by the following result.
\begin{pro}[Theorem 4.1 in~\cite{BB}, Theorem~2.1 in~\cite{CG}]\label{le prop 1}
Given two probability measures, $\mu_1$ and $\mu_2$, on $\R$, there exist unique analytic functions, $\omega_1,\omega_2\,:\,\C^+\rightarrow \C^+$, such that,
 \begin{itemize}[noitemsep,topsep=0pt,partopsep=0pt,parsep=0pt]
  \item[$(i)$] for all $z\in \C^+$, $\im \omega_1(z),\,\im \omega_2(z)\ge \im z$, and
  \begin{align}\label{le limit of omega}
  \lim_{\eta\nearrow\infty}\frac{\omega_1(\ii\eta)}{\ii\eta}=\lim_{\eta\nearrow\infty}\frac{\omega_2(\ii\eta)}{\ii\eta}=1\,;
  \end{align}
  \item[$(ii)$] for all $z\in\C^+$, 
  \begin{align}\label{le definiting equations}
   F_{\mu_1}(\omega_{2}(z))=F_{\mu_2}(\omega_{1}(z)),\qquad\qquad \omega_1(z)+\omega_2(z)-z=F_{\mu_1}(\omega_{2}(z)).
  \end{align}
 \end{itemize}
\end{pro}

It follows from~\eqref{le limit of omega} that the analytic function $F\,:\,\C^+\rightarrow \C^+$ defined by
\begin{align}\label{le kkv}
 F(z)\deq F_{\mu_1}(\omega_{2}(z))=F_{\mu_2}(\omega_{1}(z)),
\end{align}
satisfies the analogue of~\eqref{le F behaviour at infinity}. Thus $F$ is the negative reciprocal Stieltjes transform of a probability measure $\mu$, called the free additive convolution of $\mu_1$ and $\mu_2$,  usually denoted by $\mu\equiv\mu_1\boxplus\mu_2$. The functions $\omega_1$ and $\omega_2$ of Proposition~\ref{le prop 1} are called {\it subordination functions} and $F$ is said to be subordinated to~$F_{\mu_1}$, respectively to $F_{\mu_2}$. To exclude trivial shifts of measures, we henceforth assume that both,~$\mu_1$ and~$\mu_2$, are supported at more than one point. Then the analytic functions $F$, $\omega_1$ and $\omega_2$ extend continuously to the real line~\cite{Bel1,Bel}. The subordination phenomenon
was first  observed by Voiculescu~\cite{Voi93} in a generic situation
and extended to full generality by Biane~\cite{Bia98}.

We next recall the notion of {\it regular bulk} of $\mu_1\boxplus\mu_2$ introduced in~\cite{BES15b}. Let
\begin{align}\label{le set U}
\mathcal{U}_{\mu_1\boxplus\mu_2}\deq\mathrm{int}\,\Big\{\text{supp} (\mu_1\boxplus\mu_2)^{\mathrm{ac}}\,\big\backslash\,\{ x\in \R\,:\, \lim_{\eta\searrow 0}F_{\mu_1\boxplus\mu_2}(x+\ii\eta)=0\} \Big\}\, ,
\end{align} 
where $\text{supp} (\mu_1\boxplus\mu_2)^{\mathrm{ac}}$ denotes the support of the absolutely continuous part of $\mu_1\boxplus\mu_2$. We denote the density function of $(\mu_1\boxplus\mu_2)^{\mathrm{ac}}$ by $f_{\mu_1\boxplus\mu_2}$. Then the {\it regular bulk} of $\mu_1\boxplus \mu_2$ is defined as
\begin{align}
 \mathcal{B}_{\mu_1\boxplus\mu_2}\deq\mathcal{U}_{\mu_1\boxplus\mu_2}\,\backslash\, \left\{x\in\mathcal{U}_{\mu_1\boxplus\mu_2}\,:\,f_{\mu_1\boxplus\mu_2}	(x)=0\right\}.
\end{align}
 In short, regular bulk is the  regime  where the density is nonzero but finite. Finally, by general results of~\cite{Bel2}, the regular bulk contains at least one open interval; see Section~2.1 in~\cite{BES15} for detail.

\subsection{Random matrix model}
Let $A\equiv A^{(N)}$ and $B\equiv B^{(N)}$ be two sequences of deterministic real diagonal matrices in $M_N(\C)$, whose {\it empirical eigenvalue distributions} are denoted by~$\mu_A$ and~$\mu_B$, respectively. More precisely,
\begin{align}\label{le empirical measures of A and B}
 \mu_{A}\deq\frac{1}{N}\sum_{i=1}^N\delta_{a_i},\qquad\qquad\mu_{B}\deq\frac{1}{N}\sum_{i=1}^N\delta_{b_i},
\end{align}
with $A=\mathrm{diag}(a_i)$, $B=\mathrm{diag}(b_i)$. The matrices $A$ and $B$ actually depend on $N$, but we omit this from our notation. Throughout the paper, we assume that
\begin{align}\label{le bounded A and B}
\|A\|, \|B\| \leq C,
\end{align}
for some positive constant $C$ uniform in $N$. Proposition~\ref{le prop 1} asserts the existence of unique analytic functions~$\omega_A$ and $\omega_B$ satisfying the analogue of~\eqref{le limit of omega} such that, for all $z\in\C^+$,
\begin{align}\label{060101}
   F_{\mu_A}(\omega_{B}(z))=F_{\mu_B}(\omega_{A}(z)),\qquad\qquad \omega_A(z)+\omega_B(z)-z=F_{\mu_A}(\omega_{B}(z)).
\end{align}

We will assume that there are deterministic probability measures $\mu_\alpha$ and $\mu_\beta$ on $\R$, neither of them being a single point mass, such that the empirical spectral distributions $
 \mu_A$ and $\mu_B$ converge weakly, as $N\to\infty$, to $\mu_\alpha$ and $ \mu_\beta$, respectively. More precisely, we assume that 
\begin{align}\label{le assumptions convergence empirical measures}
\dL(\mu_A,\mu_\alpha)+\dL(\mu_B,\mu_\beta)\to 0,
\end{align}
as $N\to \infty$, where $\dL$ denotes the L\'evy distance. Proposition~\ref{le prop 1} asserts that there are unique analytic functions $\omega_\alpha$, $\omega_\beta$  satisfying the analogue of~\eqref{le limit of omega} such that, for all $z\in\C^+$,
\begin{align}\label{060102}
F_{\mu_\alpha}(\omega_{\beta}(z))=F_{\mu_\beta}(\omega_{\alpha}(z)),\qquad \omega_\alpha(z)+\omega_\beta(z)-z=F_{\mu_\alpha}(\omega_{\beta}(z)).
\end{align}
Proposition~4.13 of~\cite{BeV93} states that $\dL(\mu_A\boxplus\mu_B,\mu_\alpha\boxplus\mu_\beta)\le \dL(\mu_A,\mu_\alpha)+\dL(\mu_B,\mu_\beta)$, \ie the free additive convolution is continuous with respect to weak convergence of measures. 

Denote by $ U(N) $ the unitary group of degree $N$. Let $U\in U(N) $ be distributed according to Haar measure (in short a {\it Haar unitary}), and consider the random matrix
\begin{eqnarray}\label{le our H}
H\equiv H^{(N)}\deq A+UBU^*.
\end{eqnarray}
Our results also hold for the real setup when $U$ is Haar distributed on the orthogonal group, $O(N)$, of degree $N$. For definiteness, we work with the complex setup in this paper.

\subsection{Statement of main results}
To state our main results, we rely on the following definition for high-probability estimates, which was first used in~\cite{EKY}.
\begin{defn}\label{definition of stochastic domination}
Let $X\equiv X^{(N)}$, $Y\equiv Y^{(N)}$ be $N$-dependent nonnegative random variables. We say that~$Y$ stochastically dominates~$X$ if, for all (small) $\epsilon>0$ and (large)~$D>0$,
\begin{align}
\P\big(X^{(N)}>N^{\epsilon} Y^{(N)}\big)\le N^{-D},
\end{align}
for sufficiently large $N\ge N_0(\epsilon,D)$, and we write $X \prec Y$.
 When
$X^{(N)}$ and $Y^{(N)}$ depend on a parameter $v\in V$ (typically an index label or a spectral parameter), then $X(v) \prec Y (v)$, uniformly in $v\in V$, means that the threshold $N_0(\epsilon,D)$ can be chosen independently of $v$. 
\end{defn}
In Appendix~\ref{the appendix A} we collected some properties of the relation $\prec$.

Let $H$ be given in~\eqref{le our H} and denote by $(\lambda_i)_{i=1}^N$ its eigenvalues. Let $\mu_H$ stand for the empirical eigenvalue distribution of $H$, \ie
\begin{align}
 \mu_H\deq\frac{1}{N}\sum_{i=1}^N\delta_{\lambda_i}.
\end{align}
Our result on the convergence rate of $\mu_H$ to $\mu_{A\boxplus B}$ in the bulk is as follows.

\begin{thm}[Convergence rate] \label{thm.convergence rate} 
Let $\mu_\alpha$ and $\mu_\beta$ be two compactly supported probability measures on~$\mathbb{R}$, and assume that neither is supported at a single point and that at least one of them is supported at more than two points. Assume that the sequence of matrices $A$ and $B$ in~(\ref{le our H}) satisfy~(\ref{le bounded A and B}). Fix any nonempty compact interval $\mathcal{I}\subset \mathcal{B}_{\mu_\alpha\boxplus \mu_\beta}$. Then there is a (small) constant $b>0$, depending only on the measures $\mu_\alpha$ and $\mu_\beta$, on the interval~$\mathcal{I}$ and on the constant $C$ in~\eqref{le bounded A and B}, such that whenever
\begin{align}\label{le b equation}
\dL(\mu_A,\mu_\alpha)+\dL(\mu_B,\mu_\beta)\le b,
\end{align}
then
\begin{align}
\sup_{\mathcal{I'}\subseteq \mathcal{I}}\Big|\mu_H(\mathcal{I}')-\mu_{A}\boxplus\mu_{B}(\mathcal{I}')\Big|\prec \frac{1}{N},
\end{align}
where the supremum ranges over all subintervals of $\mathcal{I}$.
\end{thm}

The proof of Theorem~\ref{thm.convergence rate} is based on an optimal {\it local law for the Stieltjes transform} of the empirical eigenvalue distribution of $H$ which is the main technical result of this paper.  Denote the {\it Green function} (or the resolvent) of $H$ and its normalized trace by
\begin{align}
G(z)\equiv G_{H}(z)\deq \frac{1}{H-z},\qquad m(z)=m_H(z)\deq \ntr G(z)=\frac{1}{N}\sum_{i=1}^NG_{ii}(z),  \qquad\quad z\in\C^+,\label{Green functions 1}
\end{align}
where $G_{ij}(z)$ are the matrix entries of $G(z)$. Note that $m_H$ is the Stieltjes transform of $\mu_H$, 
\begin{align}\label{link g s}
 m_H(z)=\ntr G_H(z)=\frac{1}{N}\sum_{i=1}^N\frac{1}{\lambda_i-z}=\int_\R\frac{\dd\mu_H(x)}{x-z},\qquad \qquad z\in\C^+.
\end{align}

To state our next result, we introduce the following domain of the spectral parameter $z$: For $b\geq a\geq 0$, and $\mathcal{I}\subset \mathbb{R}$, let
\begin{align}
\mathcal{S}_{\mathcal{I}}(a,b)\deq \big\{ z=E+\mathrm{i}\eta\in \mathbb{C}^+: E\in \mathcal{I},\,	 a<\eta\leq b \big\}.  \label{le domain S}
\end{align}
Throughout the paper, we use the control parameter
 \begin{align*}
 \Psi\equiv \Psi(z)\deq \frac{1}{\sqrt{N\eta}},\quad\qquad z=E+\mathrm{i}\eta\in\C^+.
 \end{align*}

We next state the local law for the Stieltjes transform of $\mu_H$.

\begin{thm}[Local law for the Stieltjes transform] \label{thm.fluctuation averaging}
Let $\mu_\alpha$, $\mu_\beta$, $A$ and $B$ satisfy the assumption of Theorem~\ref{thm.convergence rate}. Fix any nonempty compact interval $\mathcal{I}\subset \mathcal{B}_{\mu_\alpha\boxplus \mu_\beta}$. Let $d_1,\ldots, d_N\in \mathbb{C}$ be any deterministic complex numbers satisfying 
\begin{align}
\max_{i\in \llbracket 1, N\rrbracket}|d_i|\leq 1.
\end{align}
Then there is a (small) constant $b>0$, depending only on the measures $\mu_\alpha$ and $\mu_\beta$, on the interval~$\mathcal{I}$ and on the constant $C$ in~\eqref{le bounded A and B}, such that whenever~\eqref{le b equation} holds, then
\begin{align}
\Big|\frac{1}{N}\sum_{i=1}^N d_i \Big(G_{ii}(z)-\frac{1}{a_i-\omega_B(z)}\Big)\Big|\prec \Psi^2,\label{fluctuation averaging}
\end{align}
holds uniformly on $\mathcal{S}_{\mathcal{I}}(0,1)$. In particular, choosing $d_i=1$ for all $i\in \llbracket 1, N\rrbracket$,
\begin{align}
\Big|m_H(z)-m_{\mu_A\boxplus \mu_B}(z)\Big|\prec \Psi^2, \label{032301}
\end{align}
holds uniformly on $\mathcal{S}_{\mathcal{I}}(0,1)$.
\end{thm} 
\begin{rem}
 The constant $b>0$ in Theorem~\ref{thm.convergence rate} and Theorem~\ref{thm.fluctuation averaging} is the same.
\end{rem}

\begin{rem}
In~(2.21) of~\cite{BES15b} we obtained the bound $|m_H(z)-m_{\mu_A\boxplus \mu_B}(z)|\prec \Psi$, uniformly on $\mathcal{S}_{\mathcal{I}}(0,1)$, under the same assumptions as above. The improvement to $\Psi^2$ in~\eqref{032301} is essentially due to the averaging of the fluctuation of $G_{ii}$'s in the normalized trace of the Green function.
\end{rem}
\begin{rem}
Note that the control parameter $\Psi(z)$ is small when the spectral parameter $z$ satisfies $\eta\gg N^{-1}$. Thus~\eqref{fluctuation averaging} and~\eqref{032301} are effective when $\eta$ is slightly above $N^{-1}$, while for even smaller $\eta$ the terms are simply estimated using monotonicity of the Green function. We further remark that the $N^\epsilon$ corrections in probability estimates $\prec$ can be improved to logarithmic corrections by pushing our estimates, yet we do not pursue this direction here.
\end{rem}
\begin{rem}
In Theorem~\ref{thm. two point masses} of Appendix~\ref{the Appendix C}, we collect the counterparts of the results in Theorem~\ref{thm.fluctuation averaging} and Theorem~\ref{thm.convergence rate} for the case that both $\mu_\alpha$ and $\mu_\beta$ are convex combinations of two points masses.  In fact, the result is exactly the same as in the
general case, unless
$\mu_\alpha=\mu_\beta$ when a possible singularity at  one particular energy $E$ needs to be incorporated
in the estimates.
\end{rem}

Theorem~\ref{thm.convergence rate} follows from Theorem~\ref{thm.fluctuation averaging}. It relies on the formula~\eqref{link g s} and a standard application of the Helffer-Sj\"{o}strand functional calculus. We omit the proof here and refer to,~\eg, Section~7.1 of~\cite{EKYY13} for a very similar argument.

\section{Preliminaries} \label{s. Preliminaries}
In this section, we collect some basic tools and necessary results from~\cite{BES15} and~\cite{BES15b}.

\subsection{Local stability of the system~(\ref{060101})} 
We first consider~\eqref{le definiting equations} in a general setting: For generic probability measures $\mu_1,\mu_2$, let $\PP_{\mu_1,\mu_2}\,:\, (\C^+)^{3}\rightarrow \C^2$ be given by
\begin{align}\label{le H system defs}
\PP_{\mu_1,\mu_2}(\omega_1,\omega_2,z)\deq\left(\begin{array}{cc} F_{\mu_1}(\omega_2)-\omega_1-\omega_2+z \\ F_{\mu_2}(\omega_1)-\omega_1-\omega_2+z \end{array}\right),
\end{align}
where $F_{\mu_1}$, $F_{\mu_2}$ are the negative reciprocal Stieltjes transforms of $\mu_1$, $\mu_2$; see~\eqref{le F definition}. 
Considering $\mu_1,\mu_2$ as fixed, the equation
\begin{align}\label{le H system}
\PP_{\mu_1,\mu_2}(\omega_1,\omega_2,z)=0,
\end{align}
is equivalent to~\eqref{le definiting equations} and, by Proposition~\ref{le prop 1}, there are unique analytic functions $\omega_1,\omega_2\,:\, \C^+\rightarrow \C^+$, $z\mapsto \omega_1(z),\omega_2(z)$ satisfying~\eqref{le limit of omega} that solve~\eqref{le H system} in terms of $z$. Choosing $\mu_1=\mu_\alpha$, $\mu_2=\mu_\beta$ Equation~\eqref{le H system} is equivalent to~\eqref{060102}; choosing $\mu_1=\mu_A$, $\mu_2=\mu_B$ it is equivalent to~\eqref{060101}.

We call the system~\eqref{le H system} {\it linearly $S$-stable at} $(\omega_1,\omega_2)$ if
\begin{align}\label{le what stable means}
 \left\|\left(\begin{array}{cc}
-1& F_{\mu_1}'(\omega_2)-1  \\
F_{\mu_2}'(\omega_1)-1& -1\\
  \end{array}\right)^{-1} \right\|\le S,
\end{align}
 for some positive constant $S$.

We recall a result from~\cite{BES15} showing that the system  $\Phi_{\mu_A,\mu_B}(\omega_A,\omega_B,z)=0$ is $S$-stable for all $z\in \mathcal{S}_{\mathcal{I}}(0,1)$. In Section~\ref{s. proof of main theorem} we will
use  Proposition~4.1 of~\cite{BES15}, where we showed that $S$-stability implies linear stability of the system in the sense that if
$$
 \PP_{\mu_A,\mu_B}(\omega_1(z),\omega_2(z),z)=\widetilde r(z)
$$
holds and $\omega_1, \omega_2$ are sufficiently close to $\omega_A$, $\omega_B$ at some
 $z_0\in \mathcal{S}_{\mathcal{I}}(0,1) $, then 
$$
 |\omega_1(z_0)-\omega_A(z_0)|\le  2S \|\widetilde{r}(z_0)\|_2,\qquad|\omega_2(z_0)-\omega_B(z_0)|\le  2S \|\widetilde{r}(z_0)\|_2.
$$

\begin{lem}[Lemma~5.1 and Corollary~5.2 of~\cite{BES15}] \label{cor.080601}
 Let $\mu_A$, $\mu_B$ be the probability measures from~\eqref{le empirical measures of A and B} satisfying the assumptions of Theorem~\ref{thm.convergence rate}. Let $\omega_A,\omega_B$ denote the associated subordination functions of~\eqref{060101}. Let $\mathcal{I}$ be the interval in~Theorem~\ref{thm.convergence rate} and assume that~\eqref{le b equation} holds. Then for  $N$ sufficiently large, the system $$\PP_{\mu_A,\mu_B}(\omega_A,\omega_B,z)=0$$ is $S$-stable with some positive constant~$S$, uniformly on $ \mathcal{S}_{\mathcal{I}}(0,1)$. Further, we have 
 \begin{align}\label{le bound on the derivata}
  \max_{z\in \mathcal{S}_{\mathcal{I}}(0,1)}|\omega_A'(z)|\le 2S,\qquad  \max_{z\in \mathcal{S}_{\mathcal{I}}(0,1)}|\omega'_B(z)|\le 2S,
 \end{align}
for $N$ sufficiently large. Moreover, there exist two strictly positive constants $K$ and $k$ such that, for $N$ sufficiently large, 
\begin{align}\label{le upper bound on omega AB}
 \max_{z\in \mathcal{S}_{\mathcal{I}}(0,1)}|\omega_A(z)|\le K,\qquad  \max_{z\in \mathcal{S}_{\mathcal{I}}(0,1)}|\omega_B(z)|\le K,
\end{align}
\begin{align}\label{le lower bound on omega AB}
\min_{z\in \mathcal{S}_{\mathcal{I}}(0,1)}\im\omega_A(z)\ge k,\qquad  \min_{z\in \mathcal{S}_{\mathcal{I}}(0,1)}\im\omega_B(z)\ge k.
\end{align}
\end{lem}

\subsection{Partial randomness decomposition} \label{s.PRD}
In the sequel, we recall some notations on the partial randomness decomposition and some related results from~\cite{BES15b}. We use a decomposition of Haar measure on the unitary groups obtained in~\cite{DS87} (see also~\cite{Mezzadri}):  For any $i\in \llbracket 1,N \rrbracket$, there exists an independent pair $(\mathbf{v}_i, U^{i})$, with $\mathbf{v}_i\in \mathcal{S}_{\mathbb{C}}^{N-1}\deq \{\mathbf{x}\in\mathbb{C}: \mathbf{x}^*\mathbf{x}=1\}$ a uniformly distributed complex unit vector and with  $U^{i}\in U(N-1)$ a Haar unitary matrix, such that
\begin{align}
 U=-\e{\mathrm{i}\theta_i}R_iU^{\langle i\rangle},\qquad\mathbf{r}_i\deq\sqrt{2}\frac{\mathbf{e}_i+\e{-\mathrm{i}\theta_i}\mathbf{v}_i}{\|\mathbf{e}_i+\e{-\mathrm{i}\theta_i}\mathbf{v}_i\|_2}, \qquad R_i\deq I-\mathbf{r}_i\mathbf{r}_i^*, \label{decomposition}
\end{align}
where $U^{\la i\ra}$ is a unitary matrix with $\mathbf{e}_i$ as its $i$th column and $U^i$ as its $(i,i)$-matrix minor, and where $\theta_i$ is the argument of the $i$-th component of $\mathbf{v}_i$. Since $U^{\la i\ra}\mathbf{e}_i=\mathbf{e}_i$, one can easily check 
\begin{align}
U\mathbf{e}_i=-\e{\mathrm{i}\theta_i}R_i\mathbf{e}_i=\mathbf{v}_i \label{0224100}
\end{align} 
using the definition of $R_i$ in~(\ref{decomposition}). Hence, $\mathbf{v}_i$ is actually the $i$-th column of $U$, and $R_i=R_i^*$ is the Householder reflection sending $\mathbf{e}_i$ to $-\e{-\mathrm{i}\theta_i}\mathbf{v}_i$. 

With the decomposition of $U$ in~(\ref{decomposition}), we can write
\begin{align*}
H=A+\wt{B}=A+R_i \widetilde{B}^{\la i\ra} R_i,
\end{align*}
for any $i\in \llbracket 1,N \rrbracket$, where we introduced the shorthand notations
\begin{align}
\widetilde B\deq UBU^*,\qquad\quad\widetilde{B}^{\la i\ra}\deq U^{\la i\ra}B \big(U^{\la i\ra}\big)^*. \label{0911401}
\end{align}
Clearly, we have $\wt{B}^{\la i\ra}\mathbf{e}_i=b_i\mathbf{e}_i$ and $\mathbf{e}_i^*\wt{B}^{\la i\ra}=b_i\mathbf{e}_i^*$.
 We further define
\begin{align}
H^{\la i\ra} \deq A+\widetilde{B}^{\la i \ra},\quad\qquad G^{\la i \ra}(z)\deq (H^{\la i\ra}-z)^{-1},\qquad\qquad z\in\C^+. \label{090820}
\end{align}
Note that $B^{\la i\ra}$, $H^{\la i\ra}$ and $G^{\la i\ra}$ are independent of $\mathbf{v}_i$. 

It is known that for the uniformly distributed complex unit vector $\mathbf{v}_i\in \mathcal{S}_{\mathbb{C}}^{N-1}$,  there exists a Gaussian vector $\wt{\mathbf{g}}_i\sim \mathcal{N}_\mathbb{C}(0,N^{-1}I_N)$ such that
\begin{align*}
\mathbf{v}_i=\frac{\widetilde{\mathbf{g}}_i}{\|\widetilde{\mathbf{g}}_i\|_2}.
\end{align*}
We further define
\begin{align}
\mathbf{g}_i\deq \e{-\mathrm{i}\theta_i}\widetilde{\mathbf{g}}_i,\qquad 
\mathbf{h}_i\deq \frac{\mathbf{g}_i}{\|\mathbf{g}_i\|_2}=\e{-\mathrm{i}\theta_i} \mathbf{v}_i, \qquad \ell_i\deq \frac{\sqrt{2}}{\|\mathbf{e}_i+\mathbf{h}_i\|_2}. \label{some notation for r}
\end{align}
Note that the components of $\mathbf{g}_i$ are independent. In addition, for $k\neq i$, $g_{ik}$ is a $N_{\mathbb{C}}(0,\frac{1}{N})$ random variables while $g_{ii}$ is a $\chi$-distributed random variable with $\mathbb{E}[g_{ii}^2]=\frac{1}{N}$. 
With the above notations, we can write the vector $\mathbf{r}_i$ defined in~(\ref{decomposition}) as
\begin{align}
\mathbf{r}_i=\ell_i(\mathbf{e}_i+\mathbf{h}_i). \label{022930}
\end{align}
Two simple estimates are
\begin{align}
\Big|\|\mathbf{g}_i\|_2-1-\frac{1}{2}\big(\|\mathbf{g}_i\|_2^2-1\big)\Big|\prec\frac{1}{N},\qquad\quad \Big|\ell_i^2-(1-g_{ii})\Big|\prec\frac{1}{N}, \label{012001}
\end{align}
where in the first estimate we used $\big|\|\mathbf{g}_i\|_2^2-1\big|\prec N^{-1/2}$ and in the second we used $\ell_i^2=({1+\mathbf{e}_i^*\mathbf{h}_i})^{-1}$; \cf~\eqref{some notation for r}. Moreover, according to~(\ref{0224100}), the fact  $R_i^2=I$, and the definition of $\mathbf{h}_i$ in~(\ref{some notation for r}),  we also have
\begin{align}
R_i\mathbf{e}_i=-\mathbf{h}_i,\qquad\quad R_i\mathbf{h}_i=-\mathbf{e}_i, \label{0224101}
\end{align}
which further imply the identities
\begin{align}\mathbf{h}_i^*\wt{B}^{\la i\ra}R_i=-\mathbf{e}_i^* \wt{B}, \qquad\quad\mathbf{e}_i^*\wt{B}^{\la i\ra} R_i=-b_i\mathbf{h}_i^*=-\mathbf{h}_i^*\wt{B}, \label{030430}
\end{align}
where in the first step of the second equation above we used the fact $\mathbf{e}_i^*\wt{B}^{\la i\ra}=b_i\mathbf{e}_i^*$.

Since $g_{ii}$ is $\chi$-distributed, rather than Gaussian as the $g_{ik}$'s, it is convenient to kick it out of many arguments in the sequel where Gaussian integration by parts is repeatedly used. To this end, we denote by $\mathring{\mathbf{g}}_i$ the vector obtained from $\mathbf{g}_i$ via replacing $g_{ii}$ by zero, \ie
\begin{align*}
\mathring{\mathbf{g}}_i\deq \mathbf{g}_i-g_{ii}\mathbf{e}_i.
\end{align*} 
Correspondingly, we set
\begin{align}
\mathring{\mathbf{h}}_i\deq \frac{\mathring{\mathbf{g}}_i}{\|\mathbf{g}_i\|_2}. \label{021911}
\end{align}

Throughout the paper, without loss of generality, we assume that
\begin{align}
\ntr A=\ntr B=0. \label{022902}
\end{align}

\subsection{Approximate subordination and weak local law} 

We next briefly discuss the approximate subordination property of the Green function. In addition to $H=A+UBU^*$, we also use 
\begin{align*}
\mathcal{H}\equiv \mathcal{H}^{(N)}\deq U^*AU+B
\end{align*}
and denote the Green function of $\mathcal{H}$ by
\begin{align}
 \mathcal{G}(z)\equiv \mathcal{G}_{\mathcal{H}}(z)\deq (\mathcal{H}-z)^{-1}, \qquad\qquad z\in \mathbb{C}^+. \label{Green functions}
\end{align}
Note that the normalized traces of the Green functions $G$ and $\mathcal{G}$ are equal,
\begin{align}
m_H(z)\deq \ntr G(z)=\ntr \mathcal{G}(z), \label{022805}
\end{align}
and agree with the Stieltjes transform of the empirical spectral measure $\mu_H$.
Recall $\widetilde B$ introduced in~\eqref{0911401}. For brevity, we set
\begin{align}
\wt{A}\deq U^*AU.\label{030540}
\end{align}

Following~\cite{BES15b}, we define the  {\it approximate subordination functions} by
\begin{align}
\omega_A^c(z)\deq z-\frac{\ntr \wt{A}\mathcal{G}(z)}{m_H(z)},\qquad \omega_B^c(z)\deq z-\frac{\ntr \wt{B}G(z)}{m_H(z)},\qquad\qquad z\in\mathbb{C}^+. \label{def of approximate subordination functions}
\end{align}
These are slight modifications of the approximate subordination functions used by Pastur and Vasilchuck in~\cite{VP} and by Kargin in~\cite{Kargin}. By cyclicity of the trace, we also have
\begin{align}
\omega_A^c(z)=z-\frac{\ntr A G(z)}{m_H(z)},\qquad\qquad z\in \mathbb{C}^+. \label{030520}
\end{align}
A simple observation from~(\ref{def of approximate subordination functions}),~(\ref{030520}) and the definition of the Green function is that
\begin{align}
-\frac{1}{m_H(z)}=z-\omega_A^c(z)-\omega_B^c(z). \label{022806}
\end{align}
This suggests that $\omega_A^c$ and $\omega_B^c$ approximately solve~\eqref{060101}. This is indeed the case as is confirmed by the next result obtained in~\cite{BES15b}. We need some more notation. For any (small) $\gamma>0$, set
 \begin{align}	
 \eta_{\mathrm{m}}\equiv \eta_{\mathrm{m}}(\gamma)\deq N^{-1+\gamma}. \label{030528}
 \end{align}

\begin{pro} \label{thm. entrywise estimate}(Theorem~2.6 and~(7.12) in~\cite{BES15b}) Suppose that the assumptions in Theorem~\ref{thm.convergence rate} and~\eqref{le b equation} hold. Fix any (small) $\gamma>0$ and recall $\eta_{\mathrm{m}}\equiv \eta_{\mathrm{m}}(\gamma)$ from~\eqref{030528}. Then we have
\begin{align}
  \big| \omega_A^c(z)-\omega_A(z)\big|\prec \Psi,\qquad\quad \big|\omega_B^c(z)-\omega_B(z)\big|\prec \Psi \label{022910}
\end{align}
and
\begin{align}
\max_{i,j\in\llbracket 1,N\rrbracket}\Big|G_{ij}(z)-\delta_{ij}\frac{1}{a_i-\omega_B(z)}\Big|\prec\Psi, \label{entrywise estimate}
\end{align}
uniformly on $\mathcal{S}_{\mathcal{I}}(\eta_{\mathrm{m}},1)$.
\end{pro}
From~\eqref{entrywise estimate}, we directly get the following non-optimal estimate by taking the normalized trace,
\begin{align}
\big|\ntr G(z)- m_{\mu_A\boxplus \mu_B}(z)\big|\prec \Psi,
\end{align}
uniformly on $\mathcal{S}_{\mathcal{I}}(\eta_{\mathrm{m}},1)$. While the estimate in~\eqref{entrywise estimate} is essentially optimal, the estimate in~\eqref{022910} is improved by the fluctuation averaging as is asserted by the next result.
\begin{thm}\label{thm: subordination functions}
  Suppose that the assumptions in Theorem~\ref{thm.convergence rate} and~\eqref{le b equation} hold. Fix (small) $\gamma>0$. Then
 \begin{align}
  |\omega_A^c(z)-\omega_A(z)|\prec \Psi^2,\qquad\quad |\omega_B^c(z)-\omega_B(z)|\prec \Psi^2
 \end{align}
hold uniformly on $\mathcal{S}_{\mathcal{I}}(\eta_{\mathrm{m}},1)$ with $\eta_{\mathrm{m}}\equiv \eta_{\mathrm{m}}(\gamma)$; see~\eqref{030528}.
\end{thm}

Next, recalling the notations introduced in Section~\ref{s.PRD}, we introduce the following key quantities
\begin{align}
S_i\equiv S_i(z)\deq \mathbf{h}_i^* \wt{B}^{\la i\ra} G\mathbf{e}_i,\qquad\quad T_i\equiv T_i(z)\deq \mathbf{h}_i^* G\mathbf{e}_i. \label{022920}
\end{align}
Note that here $S_i$, $T_i$ are slightly different from the counterparts in (5.1) of~\cite{BES15b}, where we used a Gaussian vector to approximate $\mathbf{h}_i$ and $1$ to approximate $\ell_i$. Such a modification of the definition does not alter the estimate on $S_i$ and $T_i$ obtained in~\cite{BES15b}; see~(\ref{022955}) below.
More specifically, we have the following lemma.
\begin{lem}\label{lem. bound for quadratic forms}  Suppose that the  assumptions in Theorem~\ref{thm.convergence rate} and~\eqref{le b equation} hold. Letting $Q_i, Q_i'$ stand for the matrix $I$ or $\wt{B}^{\la i\ra}$, and  letting $\boldsymbol{\alpha}_i,\boldsymbol{\beta}_i$ stand for $\mathbf{h}_i$ or $\mathbf{e}_i$. Fix any (small) $\gamma>0$ and recall $\eta_{\mathrm{m}}\equiv \eta_{\mathrm{m}}(\gamma)$ from~\eqref{030528}. Then, we have the bound
\begin{align}
\max_{i\in \llbracket 1, N\rrbracket} \big|\boldsymbol{\alpha}_i^* Q_i G(z) Q_i'\boldsymbol{\beta}_i\big|\prec 1
\label{030321}
\end{align} 
uniformly on $\mathcal{S}_{\mathcal{I}}(\eta_{\mathrm{m}},1)$.
For $S_i$ and $T_i$, we have the more precise estimates 
\begin{align}
\max_{i\in \llbracket 1,N \rrbracket} \Big|S_i(z)+\frac{z-\omega_B(z)}{a_i-\omega_B(z)}\Big|\prec \Psi,\qquad \quad\max_{i\in \llbracket 1, N\rrbracket}\big| T_i\big|\prec \Psi  \label{022955}
\end{align}
uniformly on $\mathcal{S}_{\mathcal{I}}(\eta_{\mathrm{m}},1)$.
\end{lem}
\begin{proof}
Using the last inequality in~(\ref{022910}) and the lower bound in~\eqref{le lower bound on omega AB}, we see that~(\ref{022955}) is equivalent~to 
\begin{align}
\max_{i\in \llbracket 1,N \rrbracket} \Big|S_i(z)+\frac{z-\omega_B^c(z)}{a_i-\omega_B^c(z)}\Big|\prec \Psi,\qquad \max_{i\in \llbracket 1, N\rrbracket}\big| T_i\big|\prec \Psi  \label{032401}.
\end{align}
The counterparts of~(\ref{030321}) and~(\ref{032401}) in~\cite{BES15b}, with $\mathbf{h}_i$ replaced by a Gaussian approximation and~$\ell_i$ replaced by~$1$ in the  quantity $\boldsymbol{\alpha}_i^* Q_i G(z) Q_i'\boldsymbol{\beta}_i$, are (5.43) and (6.3)  of~\cite{BES15b}, respectively. Hence, it suffices to show that the replacement of $\mathbf{h}_i$ by its Gaussian approximation in~\cite{BES15b} and $\ell_i$ by $1$ in the quantity $\boldsymbol{\alpha}_i^* Q_i G(z) Q_i'\boldsymbol{\beta}_i$ only causes an error of order $\Psi$. This estimate was obtained in Lemma~4.1 of~\cite{BES15b} for the case $\boldsymbol{\alpha}_i=\boldsymbol{\beta}_i=\mathbf{e}_i$ and $Q_i=Q_i'=I$, \ie $\boldsymbol{\alpha}_i^* Q_i G(z) Q_i'\boldsymbol{\beta}_i=G_{ii}$. For the other choices of $\boldsymbol{\alpha}_i,\boldsymbol{\beta}_i, Q_i$ and~$Q_i'$, the proof is nearly the same. We leave the details to the reader.
\end{proof}

\section{Proof of Theorem~\ref{thm.fluctuation averaging}} \label{s. proof of main theorem}
In this section, we prove Theorem~\ref{thm.fluctuation averaging} with the aid of the following Proposition~\ref{pro.fluctuation averaging with Upsilon}, which will be proved in Section~\ref{s. proof of pro.fluctuation averaging with Upsilon}. We introduce the tracial quantity $\Upsilon$ by setting
\begin{align}
\Upsilon\equiv\Upsilon(z)\deq \ntr \big(\wt{B}G\big)-\big( \ntr \big(\wt{B}G\big)\big)^2+\ntr G\;\ntr\big(\wt{B}G\wt{B}\big). \label{definition of Upsilon}
\end{align}
Fix a (small) $\gamma>0$. Using the identities
\begin{align}
\wt{B}G=I-(A-z)G,\qquad \wt{B}G\wt{B}=\wt{B}-A+z+(A-z)G(A-z), \label{021701}
\end{align}
and the estimate in~(\ref{entrywise estimate}), it is straightforward to check the {\it a priori} bound 
\begin{align}
|\Upsilon|\prec \Psi, \label{022904}
\end{align}
uniformly on $\mathcal{S}_{\mathcal{I}}(\eta_{\mathrm{m}},1)$, with $\eta_{\mathrm{m}}$ as in~\eqref{030528}. Theorem~\ref{thm.fluctuation averaging} then follows from the following key estimate.
\begin{pro} \label{pro.fluctuation averaging with Upsilon} Suppose that the assumptions in Theorem~\ref{thm.convergence rate} and~\eqref{le b equation} hold. Fix any (small) $\gamma>0$. Then,
\begin{align}
\bigg|\frac{1}{N}\sum_{i=1}^N d_i \bigg(G_{ii}(z)- \frac{1}{a_i-\omega_B^c(z)-\frac{\Upsilon(z)}{\ntr G(z)}}\bigg)\bigg|\prec\Psi^2, \label{fluctuation averaging with Upsilon}
\end{align}
uniformly on $\mathcal{S}_{\mathcal{I}}(\eta_{\mathrm{m}},1)$ with $\eta_{\mathrm{m}}\equiv \eta_{\mathrm{m}}(\gamma)$. By switching the r\^oles 
of $A$ and $B$, a similar statement holds
for $\mathcal{G}_{ii}$ defined in~(\ref{Green functions}) 
if $a_i$ and $\omega_B^c$ are replaced with $b_i$ and $\omega_A^c$, respectively.
\end{pro} 
With Proposition~\ref{pro.fluctuation averaging with Upsilon}, we prove Theorem~\ref{thm.fluctuation averaging} and Theorem~\ref{thm: subordination functions} at once.

\begin{proof}[Proof of Theorem~\ref{thm.fluctuation averaging} and Theorem~\ref{thm: subordination functions}]
Fix a (small) $0<\gamma<1/2$. Recall the {\it a priori} bound of $\Upsilon$ in~(\ref{022904}). First, with Proposition~\ref{pro.fluctuation averaging with Upsilon}, we show that the improved bound
\begin{align}
|\Upsilon|\prec \Psi^2 \label{022801}
\end{align}
holds uniformly on $\mathcal{S}_{\mathcal{I}}(\eta_{\mathrm{m}},1)$.  Using the identities in~(\ref{021701}), the convention~(\ref{022902}),  the {\it a priori} bound~(\ref{022904}), and the bound~(\ref{fluctuation averaging with Upsilon}) with $d_i=1$, $a_i-z$ and $(a_i-z)^2$ in the estimate of $\ntr G$, $\ntr (\wt{B}G)$ and $\ntr (\wt{B}G\wt{B})$, respectively,  
we get 
\begin{align}
\ntr G=&\ntr \big(A-\omega_B^c-\frac{\Upsilon}{\ntr G}\big)^{-1}+O_\prec\big(\Psi^2\big)\nonumber\\
=&\frac{1}{N}\sum_{i=1}^N\frac{1}{a_i-\omega_B^c}+\frac{1}{N}\sum_{i=1}^N\frac{1}{(a_i-\omega_B^c)^2}\frac{\Upsilon}{\ntr G}+O_\prec\big(\Psi^2\big)\nonumber\\
=& m_A(\omega_B^c)+m'_A(\omega_B^c)\frac{\Upsilon}{\ntr G}+O_\prec\big(\Psi^2\big),\label{estimate of tracial quantities 1}\\
\ntr (\wt{B}G)=&1-\ntr \Big((A-z)\big(A-\omega_B^c-\frac{\Upsilon}{\ntr G}\big)^{-1}\Big)+O_\prec\big(\Psi^2\big)\nonumber\\
=&1-\frac{1}{N}\sum_{i=1}^N\frac{a_i-z}{a_i-\omega_B^c}-\frac{1}{N}\sum_{i=1}^N\frac{a_i-z}{(a_i-\omega_B^c)^2}\frac{\Upsilon}{\ntr G}+O_\prec\big(\Psi^2\big)\nonumber\\
=&(z-\omega_B^c) m_A(\omega_B^c)-\Big(m_A(\omega_B^c)+(\omega_B^c-z)m'_A(\omega_B^c)\Big)\frac{\Upsilon}{\ntr G}+O_\prec\big(\Psi^2\big),\label{estimate of tracial quantities 2}\\
\ntr (\wt{B}G\wt{B})=&z+\ntr \Big((A-z)^2\big(A-\omega_B^c-\frac{\Upsilon}{\ntr G}\big)^{-1}\Big)+O_\prec\big(\Psi^2\big)\nonumber\\
=&z+\frac{1}{N}\sum_{i=1}^N\frac{(a_i-z)^2}{a_i-\omega_B^c}+\frac{1}{N}\sum_{i=1}^N\frac{(a_i-z)^2}{(a_i-\omega_B^c)^2}\frac{\Upsilon}{\ntr G}+O_\prec\big(\Psi^2\big)\nonumber\\
=&\omega_B^c-z+(\omega_B^c-z)^2 m_A(\omega_B^c)\nonumber\\
&\qquad+\Big(1+2(\omega_B^c-z) m_A(\omega_B^c)+(\omega_B^c-z)^2 m_A'(\omega_B^c)\Big)\frac{\Upsilon}{\ntr G}+O_\prec\big(\Psi^2\big),\label{estimate of tracial quantities}
\end{align}
where we also used $|a_i-\omega_B^c(z)|^{-1}\leq (\Im \omega_B^c(z))^{-1}\prec 1$ that follows from the  facts $|\omega_B^c(z)-\omega_B(z)|\prec \Psi$ and $\Im \omega_B(z)\geq k$ uniformly on $\mathcal{S}_{\mathcal{I}}(\eta_{\mathrm{m}},1)$ from~(\ref{022910}) and~(\ref{le lower bound on omega AB}), respectively. Here, $m_A'(z)$ denotes the derivative with respect to $z$ of $m_A(z)$.

Recall the definition of $\Upsilon$ in~(\ref{definition of Upsilon}). Using~(\ref{estimate of tracial quantities 1})--(\ref{estimate of tracial quantities}) and the {\it a priori} bound $|\Upsilon|\prec\Psi$ of~(\ref{022904}), we~write
\begin{align}
\Upsilon =\ntr \big(\wt{B}G\big)-\big( \ntr \big(\wt{B}G\big)\big)^2+\ntr G\;\ntr\big(\wt{B}G\wt{B}\big)=: C_1+C_2\frac{\Upsilon}{\ntr G}+O_\prec\big(\Psi^2\big),\label{equation of Upsilon}
\end{align}
where $C_1\equiv C_1(z)$ and $C_2\equiv C_2(z)$ are coefficients collected from~(\ref{estimate of tracial quantities 1})--(\ref{estimate of tracial quantities}).
It is easy to check~that
\begin{align*}
C_1(z)=&(z-\omega_B^c) m_A(\omega_B^c)-(z-\omega_B^c)^2 \big(m_A(\omega_B^c)\big)^2+m_A(\omega_B^c)\Big(\omega_B^c-z+(\omega_B^c-z)^2 m_A(\omega_B^c)\Big)=0,
\end{align*}
and 
\begin{align*}
C_2(z)=&-\Big(m_A(\omega_B^c)+(\omega_B^c-z)m'_A(\omega_B^c)\Big)+2(z-\omega_B^c) m_A(\omega_B^c)\Big(m_A(\omega_B^c)+(\omega_B^c-z)m'_A(\omega_B^c)\Big)\nonumber\\
&\qquad+m_A(\omega_B^c)\Big(1+2(\omega_B^c-z) m_A(\omega_B^c)+(\omega_B^c-z)^2 m_A'(\omega_B^c)\Big)\nonumber\\
&\qquad+m'_A(\omega_B^c)\Big(\omega_B^c-z+(\omega_B^c-z)^2 m_A(\omega_B^c)\Big)=0,
\end{align*}
for all $z\in\C^+$, \ie $C_1$ and $C_2$ vanish identically. Hence, from~\eqref{equation of Upsilon} we verified~(\ref{022801}).

Now, applying~(\ref{022801}), the facts $|\omega_B^c(z)-\omega_B(z)|\prec \Psi$, and $\Im \omega_B(z)\geq k$ uniformly on $\mathcal{S}_{\mathcal{I}}(\eta_{\mathrm{m}},1)$ from~(\ref{022910}) and~(\ref{le lower bound on omega AB}), we see from~(\ref{fluctuation averaging with Upsilon}) that  
\begin{align}
\bigg|\frac{1}{N}\sum_{i=1}^N d_i \bigg(G_{ii}(z)- \frac{1}{a_i-\omega_B^c(z)}\bigg)\bigg|\prec\Psi^2. \label{022802}
\end{align}
Switching the r\^oles of $A$ and $B$, $U$ and $U^*$, we also have
\begin{align}
\bigg|\frac{1}{N}\sum_{i=1}^N d_i \bigg(\mathcal{G}_{ii}(z)- \frac{1}{b_i-\omega_A^c(z)}\bigg)\bigg|\prec\Psi^2, \label{022803}
\end{align}
where $\mathcal{G}$ is defined in ~(\ref{Green functions}).

Setting $d_i$ to be $1$ for all $i\in \llbracket1, N\rrbracket$ in~(\ref{022802}) and~(\ref{022803}), and using~(\ref{022805}),  we obtain
\begin{align}
m_H(z)-m_A(\omega_B^c(z))=O_\prec(\Psi^2),\qquad m_H(z)-m_B(\omega_A^c(z))=O_\prec(\Psi^2). \label{022808}
\end{align}
Recalling~(\ref{022806}) and applying the {\it a priori} estimate on $\omega_A^c(z)$ and $\omega_B^c(z)$ in~(\ref{022910})  and the lower bound for $\Im \omega_A(z)$ and $\Im \omega_B(z)$ in~(\ref{le lower bound on omega AB}), we can rewrite~(\ref{022808}) as 
\begin{align*}
\big\|\Phi_{\mu_A,\mu_B}(\omega_A^c, \omega_B^c,z)\big \|_2\prec \Psi^2,
\end{align*}
where $\Phi_{\mu_A,\mu_B}$ is defined in~(\ref{le H system defs}).
Then, by Proposition 4.1 of~\cite{BES15}, we have the improved bound
\begin{align}
|\omega_A^c(z)-\omega_A(z)|\prec \Psi^2,\qquad |\omega_B^c(z)-\omega_B(z)|\prec \Psi^2. \label{022810}
\end{align}
This completes the proof of Theorem~\ref{thm: subordination functions}. 

Applying~(\ref{022810}) to~(\ref{022802}), we further get~(\ref{fluctuation averaging}) on $\mathcal{S}_{\mathcal{I}}(\eta_{\mathrm{m}},1)$. To extend the conclusion to all of~$\mathcal{S}_{\mathcal{I}}(0,1)$, we use the monotonicity of the Green function. Since $G'_{ii}(z)=\sum_{k=1}^N G_{ik}(z)G_{ki}(z)$, we have
\begin{align*}
 |G'_{ii}(z)|\le \sum_{k=1}^N |G_{ik}(z)|^{2}=\frac{\im G_{ii}(z)}{\eta}\,,
\end{align*}
as follows from the spectral decomposition of $H$. Note next that the function $s\rightarrow s \im G_{ii}(E+\ii s)$ is monotone increasing. Thus for any $\eta\in(0,\eta_{\mathrm{m}}]$, we have
\begin{align}\label{le mono 1}
 |d_{i}G_{ii}(E+\ii\eta)-d_{i}G_{ii}(E+\ii\eta_{m})|&\le |d_{i}|\int_{\eta}^{\eta_{\mathrm{m}}}\frac{s\im G_{ii}(E+\ii s)}{s^2}\,\dd s\nonumber\\
 &\le 2|d_{i}| \frac{\eta_{\mathrm{m}}}{\eta} \im G_{ii}(E+\ii\eta_{\mathrm{m}})\le C\frac{N^\gamma}{N\eta}\le CN^\gamma\Psi^2\,,
\end{align}
with high probability, for any $E\in\mathcal{I}$, where we used Proposition~\ref{thm. entrywise estimate} to bound $\im G_{ii}(z)\prec 1$, $z\in\mathcal{S}_{\mathcal{I}}(\eta_\mathrm{m},1)$. On the other hand, by~Lemma~\ref{cor.080601}, $\omega_A'(z)$ is uniformly bounded from above on $\mathcal{S}_{\mathcal{I}}(0,1)$ and $|a_i-\omega_B(z)|$ is uniformly bounded from below on $\mathcal{S}_{\mathcal{I}}(0,1)$. Thus 
\begin{align}\label{le mono 2}
 \Big|d_{i}\frac{1}{a_i-\omega_A(E+\ii\eta)}-d_{i}\frac{1}{a_i-\omega_A(E+\ii\eta_{\mathrm{m}})}\Big|\le C (\eta_\mathrm{m}-\eta)\le\Psi^2,\qquad \eta\in(0,\eta_{\mathrm{m}}],\qquad E\in\mathcal{I},
\end{align}
since $\gamma<1/2$. Hence, from~\eqref{le mono 2} and~\eqref{le mono 1}, we conclude by triangle inequality that~\eqref{fluctuation averaging} holds uniformly on $\mathcal{S}_{\mathcal{I}}(0,1)$ since it holds on $\mathcal{S}_{\mathcal{I}}(\eta_{\mathrm{m}},1)$. This proves~\eqref{fluctuation averaging} and concludes the proof of Theorem~\ref{thm.fluctuation averaging}.
\end{proof}

\section{Proof of Proposition~\ref{pro.fluctuation averaging with Upsilon}} \label{s. proof of pro.fluctuation averaging with Upsilon} 

In this section, we prove Proposition~\ref{pro.fluctuation averaging with Upsilon}, assuming the validity of Lemma~\ref{pro. estimate of high order moment} below, whose proof is postponed to Section~\ref{s.proof of pro. estimate of high order moment}. Let us introduce the notation
\begin{align}
Z_i\deq (\wt{B}G)_{ii}\ntr G-G_{ii}(\ntr \wt{B}G-\Upsilon).  \label{def of Z}
\end{align}
We have the following lemma.
\begin{lem} \label{pro. estimate of high order moment}Suppose that the assumptions in Theorem~\ref{thm.convergence rate} and~\eqref{le b equation} hold. Then, for any fixed integer $p\geq 2$, we have
\begin{align}
\mathbb{E}\Big[\Big|\frac{1}{N}\sum_{i=1}^N d_i Z_i\Big|^{2p}\Big]\prec \Psi^{4p}, \label{high order moment estimate}
\end{align}
uniformly on $\mathcal{S}_{\mathcal{I}}(\eta_{\mathrm{m}},1)$.
\end{lem}
Next, we prove Proposition~\ref{pro.fluctuation averaging with Upsilon}, with the aid of Lemma~\ref{pro. estimate of high order moment}.
\begin{proof}[Proof of Proposition~\ref{pro.fluctuation averaging with Upsilon}]
Recall the definition of $\omega_B^c(z)$ in~(\ref{def of approximate subordination functions}). 
Using the identity 
\begin{align}
(a_i-z)G_{ii}=-(\wt{B}G)_{ii}+1, \label{032001}
\end{align} 
we can write
\begin{align}
\frac{1}{N}\sum_{i=1}^N d_i \bigg(G_{ii}(z)-\frac{1}{a_i-\omega_B^c(z)-\frac{\Upsilon}{\ntr G}}\bigg)=\frac{1}{N}\sum_{i=1}^N d_i\frac{G_{ii}(\ntr \wt{B}G-\Upsilon)-(\wt{B}G)_{ii}\ntr G}{(a_i-z)\ntr G+\ntr \wt{B}G-\Upsilon}. \label{021901}
\end{align}
Recall the definition of $\Upsilon$ in~(\ref{definition of Upsilon}) and the {\it a priori} bound~(\ref{022904}). Using~(\ref{032001}),~(\ref{021701}) and~(\ref{entrywise estimate}), it is straightforward to check that
\begin{align}
\big|\ntr \wt{B}G-(z-\omega_B)m_A(\omega_B)\big|\prec\Psi,\qquad
|Z_i|=\big|G_{ii}(\ntr \wt{B}G-\Upsilon)-(\wt{B}G)_{ii}\ntr G\big|\prec\Psi.  \label{022915}
\end{align}
Hence, using~(\ref{022910}),~(\ref{022904}) and~(\ref{022915}), we obtain from~(\ref{021901}) that 
\begin{align*}
&\frac{1}{N}\sum_{i=1}^N d_i \bigg(G_{ii}(z)-\frac{1}{a_i-\omega_B^c(z)-\frac{\Upsilon}{\ntr G}}\bigg)\nonumber\\
&\qquad=\frac{1}{N}\sum_{i=1}^N \frac{d_i}{(a_i-\omega_B)m_A(\omega_B)}\Big(G_{ii}(\ntr \wt{B}G-\Upsilon)-(\wt{B}G)_{ii}\ntr G\Big)+O_\prec\big(\Psi^2\big)\nonumber\\
&\qquad= \frac{1}{N}\sum_{i=1}^N \frac{d_i}{(\omega_B-a_i)m_A(\omega_B)}Z_i+O_\prec\big(\Psi^2\big).
\end{align*}
From Lemma~\ref{cor.080601} we have $\Im \omega_B(z)\geq k$ and $m_A (\omega_B) \gtrsim 1$ uniformly on $\mathcal{S}_{\mathcal{I}}(0, 1)$, which imply 
$$\Big|\frac{d_i}{(\omega_B-a_i)m_A(\omega_B)}\Big|\lesssim 1,$$
uniformly on $\mathcal{S}_{\mathcal{I}}(0, 1)$. Thus to prove~\eqref{fluctuation averaging}, we need to show that, for any deterministic numbers  $\widetilde d_1,\ldots, \widetilde d_N\in \mathbb{C}$ satisfying $\max_i|\widetilde d_i|\leq 1$, 
\begin{align}
\bigg|\frac{1}{N}\sum_{i=1}^N \widetilde d_iZ_i\bigg|\prec \Psi^2 \label{022815}
\end{align}
holds uniformly on $\mathcal{S}_{\mathcal{I}}(\eta_{\mathrm{m}},1)$.

For fixed $z\in \mathcal{S}_{\mathcal{I}}(\eta_{\mathrm{m}},1)$, the estimate~\eqref{022815} follows from Lemma~\ref{pro. estimate of high order moment} and Markov's inequality. To get a uniform bound on $\mathcal{S}_{\mathcal{I}}(\eta_{\mathrm{m}},1)$, we choose $|\mathcal{I}| N^{8}$~lattice points $z_1, z_2,\ldots, z_{|\mathcal{I}|N^{8}}$ in $\mathcal{S}_{\mathcal{I}}(\eta_{\mathrm{m}},1)$ such that for any $ z \in \mathcal{S}_{\mathcal{I}}(\eta_{\mathrm{m}},1)$ there exists $z_n$ satisfying $|z-z_n| \leq N^{-4}$. Then using the Lipschitz continuity of  $Z_i(z)$ in $z$  with Lipschitz constant bounded by $C\eta^{-3}$, for $C$ sufficiently large, and using~\eqref{022815} for all lattice points we get~\eqref{022815} uniformly on $\mathcal{S}_{\mathcal{I}}(\eta_{\mathrm{m}},1)$ from a union bound. This completes the proof of Proposition~\ref{pro.fluctuation averaging with Upsilon}.
\end{proof}

\section{Proof of Lemma~\ref{pro. estimate of high order moment}} \label{s.proof of pro. estimate of high order moment}
In this section, we prove Lemma~\ref{pro. estimate of high order moment}. Let $Z_i$ and $d_i$ be as in Lemma~\ref{pro. estimate of high order moment}. For $k,l\in\N$, set
\begin{align}\label{the q}
\mathfrak{q}(k,l)\deq \Big(\frac{1}{N}\sum_{i=1}^N d_iZ_i\Big)^{k}\Big({\frac{1}{N}\sum_{i =1}^N\overline{ d_iZ_i}}\Big)^{l}\,.
\end{align}

To prove Lemma~\ref{pro. estimate of high order moment} we then need to show that $\E[\mathfrak{q}(p,p)]\prec \Psi^{4p}$, uniformly on $\mathcal{S}_{\mathcal{I}}(\eta_{\mathrm{m}},1)$. This is accomplished by using a recursive estimate for $\E[\mathfrak{q}(p,p)]$, see~Proposition~\ref{lem.021902} below. The use of recursive moment estimates for the fluctuation averaging mechanism was introduced in~\cite{LS16}.

In the rest of the paper, we use the following convention:  the notation $O_\prec(\Psi^k)$, for any given positive integer $k$, stands for a generic (possibly) $z$-dependent random variable $X\equiv X(z)$ that satisfies 
\begin{align}
X\prec  \Psi^k\qquad \mbox{and} \qquad\mathbb{E}[|X|^q]\prec \Psi^{qk}, \label{032601}
\end{align}
for any given positive integer $q$. In the earlier works, the notation $O_\prec(\Psi^k)$ referred
only to the first bound, $X\prec  \Psi^k$, but in this paper it is convenient to require the second one as well.
Nevertheless, in 
  the sequel, we usually only check the first bound in~(\ref{032601}) for various $X$'s. It will be clear that the second bound in~(\ref{032601}) follows from the first one in all our applications. The reason is that
   the random variables~$X$ to be estimated below are either bounded by $O(\eta^{-k_1})  = O(N^{k_1})$ for some nonnegative constant $k_1$ deterministically,  or finite products of quadratic forms of the form 
\begin{align*}
f(z) \boldsymbol{\alpha}^* Q(z) \boldsymbol{\beta}.
\end{align*}
Here $f(z): \mathbb{C}^+\to\mathbb{C}$ is a generic function satisfying $|f(z)|\leq C\eta^{-{k_2}}$ and $Q(z):\mathbb{C}^+\to M_N(\mathbb{C})$ satisfying $\|Q\|\leq C\eta^{-k_3}$ for some finite positive constants $C$, $k_2$ and $k_3$, and where   $\boldsymbol{\alpha}$ and $\boldsymbol{\beta}$ are either Gaussian or deterministically bounded in the $\|\cdot\|_2$-norm. Then it is elementary to get the second bound in ~(\ref{032601}) from the first one by using the definition of $\prec$ in~(\ref{definition of stochastic domination}) together with the above deterministic bounds or the Gaussian tail of  $\boldsymbol{\alpha}$ or $\boldsymbol{\beta}$.

Our main aim in this section is to show the following proposition.
\begin{pro}(Recursive moment estimate) \label{lem.021902}Suppose that the assumptions in Theorem~\ref{thm.convergence rate} and~\eqref{le b equation} hold. For any fixed integer $p\geq 2$, we have
\begin{align}
\E\big[\mathfrak{q}(p,p)\big]=\mathbb{E}\big[O_\prec(\Psi^2) \mathfrak{q}(p-1,p)\big]+\mathbb{E}\big[O_\prec (\Psi^4) \mathfrak{q}(p-2,p)\big]+\mathbb{E}\big[O_\prec (\Psi^4) \mathfrak{q}(p-1,p-1)\big]. \label{030410}
\end{align}
\end{pro}

\begin{proof}[Proof of Proposition~\ref{lem.021902}] According to~(\ref{0224101}), we see that $\mathbf{e}_i^*R_i=-\mathbf{h}_i^*$. Hence, 
using the decomposition~(\ref{decomposition}) with~(\ref{022930}) and recalling the notations defined in~(\ref{0911401}) and~(\ref{022920}), we have
\begin{align}
(\wt{B}G)_{ii}&=\mathbf{e}_i^*R_i\wt{B}^{\la i\ra} R_i G\mathbf{e}_i=-\mathbf{h}_i^*\wt{B}^{\la i\ra} R_iG\mathbf{e}_i\nonumber\\
&=-\mathbf{h}_i^* \widetilde{B}^{\langle i\rangle}\big( I-\ell_i^2\mathbf{e}_i\mathbf{e}_i^*-\ell_i^2\mathbf{h}_i\mathbf{e}_i^* -\ell_i^2\mathbf{e}_i\mathbf{h}_i^*-\ell_i^2\mathbf{h}_i\mathbf{h}_i^*\big) G\mathbf{e}_i\nonumber\\
&=-S_i+ \ell_i^2 \big(\mathbf{h}_i^*\wt{B}^{\la i\ra}\mathbf{e}_i+ \mathbf{h}_i^*\wt{B}^{\la i\ra}\mathbf{h}_i\big)\big(G_{ii}+T_i\big). \label{022110}
\end{align}
Now, recalling the definition of $\mathbf{h}_i$ in ~(\ref{some notation for r}) and using the large deviation inequalities in ~(\ref{091731}), we get 
\begin{align}
|\mathbf{h}_i^* \wt{B}^{\la i\ra}\mathbf{e}_i|\prec\frac{1}{\sqrt{N}},\qquad\quad |\mathbf{h}_i^*\wt{B}^{\la i\ra}\mathbf{h}_i|\prec\frac{1}{\sqrt{N}}, \label{021910}
\end{align}
where we also used the convention $\ntr \wt{B}^{\la i\ra}=\ntr B=0$ from~(\ref{022902}). According to Lemma~\ref{lem. bound for quadratic forms},~(\ref{entrywise estimate}) and Lemma~\ref{cor.080601}, 
we also have
\begin{align}
|T_i|\prec\Psi,\qquad\quad |G_{ii}|\prec 1. \label{030401}
\end{align}
In addition, by~(\ref{012001}), we have the elementary estimate
\begin{align}
\ell_i=1+O_\prec(\frac{1} {\sqrt{N}}). \label{estimate of l}
\end{align}
Now, using~(\ref{021910})-(\ref{estimate of l}) to bound several small terms in~(\ref{022110}), we obtain
\begin{align}
(\wt{B}G)_{ii}=-S_i+\mathbf{h}_i^*\wt{B}^{\la i\ra}\mathbf{e}_iG_{ii}+\mathbf{h}_i^*\wt{B}^{\la i\ra}\mathbf{h}_iG_{ii}+O_\prec(\Psi^2).\label{expansion of BG_ii}
\end{align}
Moreover, using the fact $\wt{B}^{\la i\ra}\mathbf{e}_i=b_i\mathbf{e}_i$,   we can write 
\begin{align}
(\wt{B}G)_{ii}=&-\sum_{k: k\neq i} \bar{h}_{ik} \mathbf{e}_k^*\wt{B}^{\la i\ra} G\mathbf{e}_i+\mathbf{h}_i^*\wt{B}^{\la i\ra}\mathbf{h}_i  G_{ii}+O_\prec(\Psi^2)\nonumber\\
=& -\mathring{\mathbf{h}}_i^*\wt{B}^{\la i\ra} G\mathbf{e}_i+\mathbf{h}_i^*\wt{B}^{\la i\ra}\mathbf{h}_i G_{ii}+O_\prec(\Psi^2), \label{021912}
\end{align}
where in the last step we used the notation introduced in~(\ref{021911}).

Recalling the definition of $Z_i$ in~(\ref{def of Z}), with~(\ref{021912}), we can write 
\begin{align}
\E\big[\mathfrak{q}(p,p)\big]=&\mathbb{E}\Big[\Big(\frac{1}{N}\sum_{i=1}^N d_iZ_i\Big) \mathfrak{q}(p-1,p)\Big]\nonumber\\
=&\mathbb{E}\Big[\Big(\frac{1}{N}\sum_{i=1}^N d_i(\wt{B}G)_{ii} \ntr G\Big) \mathfrak{q}(p-1,p)\Big]-\mathbb{E}\Big[\Big(\frac{1}{N}\sum_{i=1}^N d_iG_{ii} (\ntr \wt{B}G-\Upsilon)\Big) \mathfrak{q}(p-1,p)\Big]\nonumber\\
=&-\mathbb{E}\Big[\Big(\frac{1}{N}\sum_{i=1}^N d_i\mathring{\mathbf{h}}_i^*\wt{B}^{\la i\ra} G\mathbf{e}_i \ntr G\Big) \mathfrak{q}(p-1,p)\Big]-\mathbb{E}\Big[\Big(\frac{1}{N}\sum_{i=1}^N d_iG_{ii} (\ntr \wt{B}G-\Upsilon)\Big) \mathfrak{q}(p-1,p)\Big]\nonumber\\
&\qquad+\mathbb{E}\Big[\Big(\frac{1}{N}\sum_{i=1}^N d_i\mathbf{h}_i^*\wt{B}^{\la i\ra}\mathbf{h}_i G_{ii} \ntr G\Big) \mathfrak{q}(p-1,p)\Big]+\mathbb{E}\Big[O_\prec(\Psi^2)\mathfrak{q}(p-1,p)\Big]. \label{021601}
\end{align}
Next, we claim that the following lemma holds.
\begin{lem} \label{lem. 022301}Suppose that the assumptions in Theorem~\ref{thm.convergence rate} and~\eqref{le b equation} hold. Then, for any fixed integer $p\geq 2$, we have
\begin{align}
&\mathbb{E}\Big[\Big(\frac{1}{N}\sum_{i=1}^N d_i\mathring{\mathbf{h}}_i^*\wt{B}^{\la i\ra} G\mathbf{e}_i \ntr G\Big) \mathfrak{q}(p-1,p)\Big]+\mathbb{E}\Big[\Big(\frac{1}{N}\sum_{i=1}^N d_iG_{ii} (\ntr \wt{B}G-\Upsilon)\Big) \mathfrak{q}(p-1,p)\Big]\nonumber\\
&\qquad=\mathbb{E}\Big[O_\prec(\Psi^2)\mathfrak{q}(p-1,p)\Big]+\mathbb{E}\Big[O_\prec(\Psi^4)\mathfrak{q}(p-2,p)\Big]+\mathbb{E}\Big[O_\prec(\Psi^4)\mathfrak{q}(p-1,p-1)\Big]. \label{part 1}
\end{align}
Similarly, we have
\begin{align}
&\mathbb{E}\Big[\Big(\frac{1}{N}\sum_{i=1}^N d_i\mathbf{h}_i^*\wt{B}^{\la i\ra}\mathbf{h}_i G_{ii} \ntr G\Big) \mathfrak{q}(p-1,p)\Big]\nonumber\\
&\qquad=\mathbb{E}\Big[O_\prec(\Psi^2)\mathfrak{q}(p-1,p)\Big]+\mathbb{E}\Big[O_\prec(\Psi^4)\mathfrak{q}(p-2,p)\Big]+\mathbb{E}\Big[O_\prec(\Psi^4)\mathfrak{q}(p-1,p-1)\Big]. \label{part 2}
\end{align}
\end{lem}
The proof of Lemma~\ref{lem. 022301} will be postponed. Combining ~(\ref{021601}),~(\ref{part 1}) and~(\ref{part 2}), we can conclude the proof of Proposition~\ref{lem.021902}.
\end{proof}

With Proposition~\ref{lem.021902}, we can prove Lemma~\ref{pro. estimate of high order moment}.
\begin{proof}[Proof of Lemma~\ref{pro. estimate of high order moment}]
Fix any (small) $\epsilon>0$. Then applying Young's inequality to~\eqref{030410} we~get
\begin{align}
 \E\big[\mathfrak{q}(p,p)\big]\le 3\frac{1}{2p}\E\big[O_\prec(N^{2 p\epsilon}\Psi^{4p})\big]+3\frac{2p-1}{2p} N^{-\frac{2p\epsilon}{2p-1}}\E\big[\mathfrak{q}(p,p)\big].
\end{align}
Hence absorbing the second term on the right side into the left side and recalling~(\ref{032601}) we get
\begin{align}
 \E\big[\mathfrak{q}(p,p)\big]\prec\Psi^{4p},
\end{align}
uniformly on $\mathcal{S}_{\mathcal{I}}(\eta_{\mathrm{m}},1)$, since $\epsilon>0$ was arbitrary.
\end{proof}

In the rest of this section, we prove Lemma~\ref{lem. 022301}.
\begin{proof}[Proof of Lemma~\ref{lem. 022301}] We use integration by parts for the Gaussian variables:  regarding $g$ and $\bar{g}$ as independent variables for computing $\partial_g f(g, \bar{g})$, we have
\begin{align}
\int_{\mathbb{C}} \bar{g} f(g, \bar{g})\, \mathrm{e}^{-\frac{|g|^2}{\sigma^2}} {\rm d} g \wedge {\rm d} \bar{g} =\sigma^2 \int_{\mathbb{C}} \partial_g f(g, \bar{g})\, \mathrm{e}^{-\frac{|g|^2}{\sigma^2}} {\rm d} g\wedge {\rm d} \bar{g}, \label{030420}
\end{align}
for differentiable functions $f: \mathbb{C}^2\to \mathbb{C}$.

 Let us start with~(\ref{part 1}). First, we can get rid of the $\mathbf{g}_i$-dependence of the factor $\ntr G$, namely, 
\begin{align}\label{smakaka}
\mathbb{E}\Big[\Big(\frac{1}{N}\sum_{i=1}^N d_i\mathring{\mathbf{h}}_i^*\wt{B}^{\la i\ra} G\mathbf{e}_i \ntr G\Big) \mathfrak{q}(p-1,p)\Big]=&\mathbb{E}\Big[\Big(\frac{1}{N}\sum_{i=1}^N d_i\mathring{\mathbf{h}}_i^*\wt{B}^{\la i\ra} G\mathbf{e}_i \ntr G^{\la i\ra}\Big) \mathfrak{q}(p-1,p)\Big]\nonumber\\
&\qquad+\mathbb{E}\Big[O_\prec(\Psi^2)\mathfrak{q}(p-1,p)\Big],
\end{align}
where we used the finite rank perturbation estimate in~(\ref{091002}) and
\begin{align}
|\mathring{\mathbf{h}}_i^*\wt{B}^{\la i\ra} G\mathbf{e}_i|\prec 1, \qquad\quad \forall i\in \llbracket 1, N\rrbracket,\label{030460}
\end{align}
 which follows from $\mathring{\mathbf{h}}_i^*\wt{B}^{\la i\ra} G\mathbf{e}_i=S_i-b_ih_{ii}G_{ii}$ and the bounds in Lemma~\ref{lem. bound for quadratic forms}.
Further, for brevity, we~let
\begin{align}
d_{i,1}\equiv d_{i,1}(z)\deq d_i \ntr G^{\la i\ra}. \label{022601}
\end{align}
Recalling the definition in~(\ref{021911}) and using the integration by parts formula~(\ref{030420}) for the Gaussian variables $\bar{g}_{ik}$, $i\not=k$, we get
\begin{align}
&\mkern-18mu\mathbb{E}\Big[\Big(\frac{1}{N}\sum_{i=1}^N d_{i,1}\mathring{\mathbf{h}}_i^*\wt{B}^{\la i\ra} G\mathbf{e}_i \Big) \mathfrak{q}(p-1,p)\Big] \nonumber\\
&=  \mathbb{E}\Big[\Big(\frac{1}{N^2}\sum_{i=1}^N d_{i,1} \frac{1}{\|\mathbf{g}_i\|_2}\sum_{k:k\neq i} \frac{\partial (\mathbf{e}_k^*\wt{B}^{\la i\ra} G\mathbf{e}_i)}{\partial g_{ik}} \Big) \mathfrak{q}(p-1,p)\Big]\nonumber\\
&\qquad+\mathbb{E}\Big[\Big(\frac{1}{N^2}\sum_{i=1}^N d_{i,1}\sum_{k:k\neq i} \frac{\partial \|\mathbf{g}_i\|_2^{-1}}{\partial g_{ik}}  (\mathbf{e}_k^*\wt{B}^{\la i\ra} G\mathbf{e}_i)\Big) \mathfrak{q}(p-1,p)\Big]\nonumber\\
&\qquad+\mathbb{E}\Big[\Big(\frac{p-1}{N^3}\sum_{i=1}^N d_{i,1} \frac{1}{\|\mathbf{g}_i\|_2}\sum_{k:k\neq i} \mathbf{e}_k^*\wt{B}^{\la i\ra} G\mathbf{e}_i \sum_{j=1}^N d_j\frac{\partial Z_j}{\partial g_{ik}}\Big)\mathfrak{q}(p-2,p)\Big]\nonumber\\
& \qquad+\mathbb{E}\Big[\Big(\frac{p}{N^3}\sum_{i=1}^N d_{i,1} \frac{1}{\|\mathbf{g}_i\|_2}\sum_{k:k\neq i} \mathbf{e}_k^*\wt{B}^{\la i\ra} G\mathbf{e}_i  \sum_{j=1}^N \overline{d_j}\frac{\partial \overline{Z_j}}{\partial g_{ik}}\Big)\mathfrak{q}(p-1,p-1)\Big]. \label{021610}
\end{align}
Using~(\ref{decomposition}) and~(\ref{022930}), it is elementary to compute
\begin{align}
\frac{\partial R_i}{\partial g_{ik}}&= -\frac{\ell_i^2}{\|\mathbf{g}_i\|_2} \mathbf{e}_k (\mathbf{e}_i+\mathbf{h}_i)^*+\frac{\ell_i^2}{2\|\mathbf{g}_i\|_2^2} \bar{g}_{ik} \big(\mathbf{e}_i\mathbf{h}_i^*+\mathbf{h}_i\mathbf{e}_i^*+2\mathbf{h}_i\mathbf{h}_i^*\big)-\frac{\ell_i^4}{2\|\mathbf{g}_i\|_2^3} g_{ii}\bar{g}_{ik}(\mathbf{e}_i+\mathbf{h}_i)(\mathbf{e}_i+\mathbf{h}_i)^*\nonumber\\
&=:  -\frac{\ell_i^2}{\|\mathbf{g}_i\|_2} \mathbf{e}_k (\mathbf{e}_i+\mathbf{h}_i)^*+\Delta_R(i,k),\label{021625}
\end{align}
where we introduced
\begin{align}
\Delta_R(i,k)\deq \frac{\ell_i^2}{2\|\mathbf{g}_i\|_2^2} \bar{g}_{ik} \big(\mathbf{e}_i\mathbf{h}_i^*+\mathbf{h}_i\mathbf{e}_i^*+2\mathbf{h}_i\mathbf{h}_i^*\big)-\frac{\ell_i^4}{2\|\mathbf{g}_i\|_2^3} g_{ii}\bar{g}_{ik}(\mathbf{e}_i+\mathbf{h}_i)(\mathbf{e}_i+\mathbf{h}_i)^*. \label{0223300}
\end{align}
The $\Delta_R(i,k)$'s are irrelevant  error terms. Their estimates will be easy and kept separate in Appendix~\ref{a.A}. We focus on the other terms in the sequel.  
For convenience, we introduce
\begin{align}
c_i\deq \frac{\ell_i^2}{\|\mathbf{g}_i\|_2}=\frac{1}{\|\mathbf{g}_i\|_2}-g_{ii}+O_\prec(\frac{1}{N})=\|\mathbf{g}_i\|_2-g_{ii}-(\|\mathbf{g}_i\|_2^2-1)+O_\prec(\frac{1}{N}), \label{022140}
\end{align}
where the last step follows from~(\ref{012001}).
Using~(\ref{021625}), we have 
\begin{align}
\frac{\partial G}{\partial g_{ik}}&=-G\frac{\partial \wt{B}}{\partial g_{ik}} G=-G\frac{\partial R_i}{\partial g_{ik}} \wt{B}^{\la i\ra} R_i G-G R_i \wt{B}^{\la i\ra} \frac{\partial R_i}{\partial g_{ik}} G\nonumber\\
&=:c_i\Big[G\mathbf{e}_k (\mathbf{e}_i+\mathbf{h}_i)^*\wt{B}^{\la i\ra} R_i G+GR_i\wt{B}^{\la i\ra}\mathbf{e}_k (\mathbf{e}_i+\mathbf{h}_i)^*G\Big]+ \Delta_G(i,k),\label{022240}
\end{align}
where we set
\begin{align}
\Delta_G(i,k)\deq -G\Delta_R(i,k) \wt{B}^{\la i\ra} R_i G-G R_i \wt{B}^{\la i\ra} \Delta_R(i,k) G. \label{022360}
\end{align}
Hence, applying~(\ref{022240}), we obtain, for any $i\in\llbracket 1,N\rrbracket$,
\begin{align}
\frac{1}{N}\sum_{k:k\neq i} \frac{\partial (\mathbf{e}_k^*\wt{B}^{\la i\ra} G\mathbf{e}_i)}{\partial g_{ik}}&=\frac{1}{N}\sum_{k:k\neq i} \mathbf{e}_k^*\wt{B}^{\la i\ra}\frac{\partial  G}{\partial g_{ik}}\mathbf{e}_i\nonumber\\
&= c_i \Big[\ntr (\wt{B}^{\la i\ra} G) (\mathbf{e}_i+\mathbf{h}_i)^* \wt{B}^{\la i\ra} R_i G\mathbf{e}_i+ \ntr \big(\wt{B}^{\la i\ra} G R_i \wt{B}^{\la i\ra}\big) (\mathbf{e}_i+\mathbf{h}_i)^*G\mathbf{e}_i\Big]\nonumber\\
&\qquad-c_i \frac{1}{N}\Big[\mathbf{e}_i^*\wt{B}^{\la i\ra}G\mathbf{e}_i (\mathbf{e}_i+\mathbf{h}_i)^*\wt{B}^{\la i\ra} R_i G\mathbf{e}_i+\mathbf{e}_i^*\wt{B}^{\la i\ra}GR_i\wt{B}^{\la i\ra}\mathbf{e}_i (\mathbf{e}_i+\mathbf{h}_i)^*G\mathbf{e}_i\Big]\nonumber\\
&\qquad+\frac{1}{N}\sum_{k:k\neq i} \mathbf{e}_k^*\wt{B}^{\la i\ra} \Delta_G(i,k)\mathbf{e}_i\nonumber\\
&=c_i \Big[-\ntr (\wt{B}^{\la i\ra} G) \big(b_iT_i+(\wt{B}G)_{ii}\big)+ \ntr \big(\wt{B}^{\la i\ra} G R_i \wt{B}^{\la i\ra}\big) \big(G_{ii}+T_i\big)\Big]\nonumber\\
&\qquad+\frac{1}{N}\sum_{k:k\neq i} \mathbf{e}_k^*\wt{B}^{\la i\ra} \Delta_G(i,k)\mathbf{e}_i+O_\prec(\Psi^2), \label{022950}
\end{align}
where in the last step we used~(\ref{030430}) and thus
\begin{align}
\hspace{-2ex}\mathbf{e}_i^*\wt{B}^{\la i\ra} R_i G\mathbf{e}_i=-b_iT_i,\quad \mathbf{h}_i^*\wt{B}^{\la i\ra} R_i G\mathbf{e}_i=-(\wt{B}G)_{ii},\quad \mathbf{e}_i^*\wt{B}^{\la i\ra}GR_i\wt{B}^{\la i\ra}\mathbf{e}_i=-b_i^2\mathbf{e}_i^*G\mathbf{h}_i, \label{022650}
\end{align}
whose bounds can be obtained from Lemma~\ref{lem. bound for quadratic forms} and the identity $(\wt{B}G)_{ii}=1-(a_i-z)G_{ii}$.

For the second term of the right side of~(\ref{022950}), we use the next lemma, which is proved in Appendix~\ref{a.A}.
\begin{lem} \label{lem.022201} Suppose that the assumptions in Theorem~\ref{thm.convergence rate} and~\eqref{le b equation} hold, we have
\begin{align}
\frac{1}{N}\sum_{k:k\neq i} \mathbf{e}_k^*\wt{B}^{\la i\ra} \Delta_G(i,k)\mathbf{e}_i=O_\prec(\Psi^2). \label{022410}
\end{align}
\end{lem}
With the aid of  Lemma~\ref{lem.022201}, we get from~(\ref{022950}) that
\begin{align}
\frac{1}{N}\sum_{k:k\neq i} \frac{\partial (\mathbf{e}_k^*\wt{B}^{\la i\ra} G\mathbf{e}_i)}{\partial g_{ik}}&= c_i \Big[ \ntr \big(\wt{B}^{\la i\ra} G R_i \wt{B}^{\la i\ra}\big) \big(G_{ii}+T_i\big)-\ntr (\wt{B}^{\la i\ra} G) \big(b_iT_i+(\wt{B}G)_{ii}\big)\Big]+O_\prec(\Psi^2)\nonumber\\
&=c_i \Big[ \ntr \big(\wt{B} G\wt{B}\big) \big(G_{ii}+T_i\big)-\ntr (\wt{B} G) \big(b_iT_i+(\wt{B}G)_{ii}\big)\Big]+O_\prec(\Psi^2),\label{021626}
\end{align}
where in the last step we used the estimates in~(\ref{030401}) and the facts that the differences $\ntr (\wt{B}G)- \ntr (\wt{B}^{\la i\ra} G)$ and $\ntr \big(\wt{B} G \wt{B}\big)- \ntr \wt{B}^{\la i\ra} G R_i\wt{B}^{\la i\ra}$ can be written as the linear combination of the terms of the form $\frac{1}{N} \mathbf{r}_i^*Q_iGQ_i'\mathbf{r}_i$ for $Q_i,Q_i'=I$ or $\wt{B}^{\la i\ra}$, which implies according to~(\ref{030321}) that
\begin{align*}
\ntr (\wt{B}G)- \ntr (\wt{B}^{\la i\ra} G)=O_\prec(\frac{1}{N}), \qquad \ntr \big(\wt{B} G \wt{B}\big)- \ntr \wt{B}^{\la i\ra} G R_i\wt{B}^{\la i\ra}=O_\prec (\frac{1}{N}).
 \end{align*} 
Analogously to~(\ref{021626}), we can get
\begin{align}
\frac{1}{N}\sum_{k:k\neq i} \frac{\partial (\mathbf{e}_k^*G\mathbf{e}_i)}{\partial g_{ik}}
=& c_i \Big[\ntr \big(\wt{B} G\big) \big(G_{ii}+T_i\big)-(\ntr G) \big(b_iT_i+(\wt{B}G)_{ii}\big)\Big]+O_\prec(\Psi^2). \label{021627}
\end{align}

Combining~(\ref{021626}) and~(\ref{021627}), and recalling the definition of $\Upsilon$ in~(\ref{definition of Upsilon}),  we obtain
\begin{align}
&\frac{1}{N}\sum_{k:k\neq i} \frac{\partial (\mathbf{e}_k^*\wt{B}^{\la i\ra} G\mathbf{e}_i)}{\partial g_{ik}} \ntr G-\frac{1}{N}\sum_{k:k\neq i} \frac{\partial (\mathbf{e}_k^*G\mathbf{e}_i)}{\partial g_{ik}} \ntr (\wt{B}G)\nonumber\\
&\qquad=c_i (G_{ii}+T_i)\Big(\ntr G\ntr \big(\wt{B} G \wt{B}\big)-\big(\ntr \wt{B}G\big)^2\Big) +O_\prec(\Psi^2)\nonumber\\
&\qquad=-c_i (G_{ii}+T_i)(\ntr \wt{B}G-\Upsilon)+O_\prec(\Psi^2). \label{021629}
\end{align}
Now, we set
\begin{align}
\mathring{T}_i\deq \mathring{\mathbf{g}}_i^*G\mathbf{e}_i=\sum_{k:k\neq i} \bar{g}_{ik} \mathbf{e}_k^*G\mathbf{e}_i=T_i-g_{ii}G_{ii}+O_\prec(\Psi^2), \label{022030}
\end{align}
where in the last step we used the definition of $T_i$ in~(\ref{022920}), the bound $|T_i|\prec \Psi$ from~(\ref{022955}), and the estimate $\|\mathbf{g}_i^2\|^{-1}=1+O_\prec(N^{-1/2})$. 
Using ~(\ref{022140}) and~(\ref{022030}),  we rewrite~(\ref{021629}) as
\begin{align}
&\hspace{-8ex}\frac{1}{N}\sum_{k:k\neq i} \frac{\partial (\mathbf{e}_k^*\wt{B}^{\la i\ra} G\mathbf{e}_i)}{\partial g_{ik}} \ntr G\nonumber\\
=&-c_i (G_{ii}+T_i)(\ntr \wt{B}G-\Upsilon) +\mathring{T}_i\ntr (\wt{B}G)+\Big(\frac{1}{N}\sum_{k:k\neq i} \frac{\partial (\mathbf{e}_k^*G\mathbf{e}_i)}{\partial g_{ik}}-\mathring{T}_i\Big) \ntr (\wt{B}G)+O_\prec(\Psi^2)\nonumber\\
=&-\Big(\|\mathbf{g}_i\|_2-g_{ii}-(\|\mathbf{g}_i\|_2^2-1)\Big)(G_{ii}+T_i)(\ntr (\wt{B}G)-\Upsilon)+\big(T_i-g_{ii}G_{ii}\big)\ntr (\wt{B}G) \nonumber\\
&\qquad+\Big(\frac{1}{N}\sum_{k:k\neq i} \frac{\partial (\mathbf{e}_k^*G\mathbf{e}_i)}{\partial g_{ik}}-\mathring{T}_i\Big) \ntr (\wt{B}G)+O_\prec(\Psi^2)\nonumber\\
=&-\|\mathbf{g}_i\|_2G_{ii}(\ntr \wt{B}G-\Upsilon)+\Big(\frac{1}{N}\sum_{k:k\neq i} \frac{\partial (\mathbf{e}_k^*G\mathbf{e}_i)}{\partial g_{ik}}-\mathring{T}_i\Big) \ntr (\wt{B}G)\nonumber\\
&\qquad+\big(\|\mathbf{g}_i\|_2^2-1\big)G_{ii}\ntr (\wt{B}G)+O_\prec(\Psi^2), \label{021631}
\end{align}
where in the last step we used the bound $|T_i|\prec \Psi$, $|\Upsilon|\prec \Psi$ from~(\ref{022955}) and~(\ref{022904}), and $|g_{ii}|\prec N^{-1/2}$ and  $\|\mathbf{g}_i\|_2=1+O_\prec (N^{-1/2})$. Notice that the two potentially dangerous terms $g_{ii} G_{ii}\ntr (\wt{B}G)$ of order $N^{-1/2}$
cancel exactly.
Recalling from~(\ref{022601}) that $d_{i,1}=d_i \ntr G^{\la i\ra}=d_i \ntr G+O(\Psi^2)$, and using~(\ref{021631}),  we have
\begin{align}
\frac{1}{N^2}\sum_{i=1}^N d_{i,1} \frac{1}{\|\mathbf{g}_i\|_2}\sum_{k:k\neq i} \frac{\partial (\mathbf{e}_k^*\wt{B}^{\la i\ra} G\mathbf{e}_i)}{\partial g_{ik}} 
& =-\frac{1}{N}\sum_{i=1}^N d_iG_{ii} (\ntr \wt{B}G-\Upsilon) \nonumber\\
&\qquad+\frac{1}{N}\sum_{i=1}^N d_{i} \frac{1}{\|\mathbf{g}_i\|_2}\Big(\frac{1}{N}\sum_{k:k\neq i} \frac{\partial (\mathbf{e}_k^*G\mathbf{e}_i)}{\partial g_{ik}}-\mathring{T}_i\Big) \ntr (\wt{B}G) \nonumber\\
&\qquad +\frac{1}{N}\sum_{i=1}^N d_{i} \frac{\big(\|\mathbf{g}_i\|_2^2-1\big)}{\|\mathbf{g}_i\|_2}G_{ii}\ntr (\wt{B}G) +O_\prec(\Psi^2). \label{021650}
\end{align}
Substituting~(\ref{021650}) into~(\ref{021610}) and recalling~\eqref{smakaka}, we obtain 
\begin{align}
&\hspace{-8ex}\mathbb{E}\Big[\Big(\frac{1}{N}\sum_{i=1}^N d_i\mathring{\mathbf{h}}_i^*\wt{B}^{\la i\ra} G\mathbf{e}_i \ntr G\Big) \mathfrak{q}(p-1,p)\Big]+\mathbb{E}\Big[\Big(\frac{1}{N}\sum_{i=1}^N d_iG_{ii} (\ntr (\wt{B}G)-\Upsilon)\Big) \mathfrak{q}(p-1,p)\Big]\nonumber\\
&=\mathbb{E}\Big[\Big(\frac{1}{N}\sum_{i=1}^N d_{i} \frac{1}{\|\mathbf{g}_i\|_2}\Big(\frac{1}{N}\sum_{k:k\neq i} \frac{\partial (\mathbf{e}_k^*G\mathbf{e}_i)}{\partial g_{ik}}-\mathring{T}_i\Big) \ntr (\wt{B}G)\Big) \mathfrak{q}(p-1,p)\Big]\nonumber\\
& \qquad+\mathbb{E}\Big[\Big(\frac{1}{N}\sum_{i=1}^N d_{i} \frac{\big(\|\mathbf{g}_i\|_2^2-1\big)}{\|\mathbf{g}_i\|_2}G_{ii}\ntr (\wt{B}G) \mathfrak{q}(p-1,p)\Big]\nonumber\\
&\qquad+\mathbb{E}\Big[\Big(\frac{1}{N^2}\sum_{i=1}^N d_{i,1}\sum_{k:k\neq i}  \frac{\partial \|\mathbf{g}_i\|_2^{-1}}{\partial g_{ik}} (\mathbf{e}_k^*\wt{B}^{\la i\ra} G\mathbf{e}_i)\Big) \mathfrak{q}(p-1,p)\Big]\nonumber\\
&\qquad+\mathbb{E}\Big[\Big(\frac{p-1}{N^3}\sum_{i=1}^N d_{i,1} \frac{1}{\|\mathbf{g}_i\|_2}\sum_{k:k\neq i} \mathbf{e}_k^*\wt{B}^{\la i\ra} G\mathbf{e}_i \sum_{j=1}^N d_j\frac{\partial Z_j}{\partial g_{ik}}\Big)\mathfrak{q}(p-2,p)\Big]\nonumber\\
&\qquad +\mathbb{E}\Big[\Big(\frac{p}{N^3}\sum_{i=1}^N d_{i,1} \frac{1}{\|\mathbf{g}_i\|_2}\sum_{k:k\neq i} \mathbf{e}_k^*\wt{B}^{\la i\ra} G\mathbf{e}_i \sum_{j=1}^N \overline{d_j}\frac{\partial \overline{Z_j}}{\partial g_{ik}}\Big)\mathfrak{q}(p-1,p-1)\Big]\nonumber\\
&\qquad+\mathbb{E}\Big[O_\prec(\Psi^2)\mathfrak{q}(p-1,p)\Big]. \label{022035}
\end{align}
Hence, for~(\ref{part 1}), it suffices to estimate the right side of~(\ref{022035}). We start with the first term on the right side of~(\ref{022035}). First, by~(\ref{021627}), one can easily check that $\frac{1}{N}\sum_{k:k\neq i} \frac{\partial (\mathbf{e}_k^*G\mathbf{e}_i)}{\partial g_{ik}}=O_\prec(\Psi)$ for any $i$ from the estimate of $G_{ii}$'s and $T_i$'s (\cf~(\ref{entrywise estimate}) and~(\ref{022955})), and the first identity in~(\ref{021701}) that expresses $(\wt{B}G)_{ii}$ in terms of $G_{ii}$. In addition, we also have $\mathring{T}_i=O_\prec(\Psi)$ from~(\ref{022030}) and the estimate of $G_{ii}$ and $T_i$ (\cf~(\ref{entrywise estimate}) and~(\ref{022955})). These facts, together with $\|\mathbf{g}_i\|_2=1+O_\prec(\frac{1}{\sqrt{N}})$ and  the finite rank perturbation bound for the tracial quantities of Green function in Corollary~\ref{cor. finite perturbation}, we have

\begin{align}
&\mkern-18mu\mkern-18mu\mathbb{E}\Big[\Big(\frac{1}{N}\sum_{i=1}^N d_i \frac{1}{\|\mathbf{g}_i\|_2}\Big(\frac{1}{N}\sum_{k:k\neq i} \frac{\partial (\mathbf{e}_k^*G\mathbf{e}_i)}{\partial g_{ik}}-\mathring{T}_i\Big) \ntr (\wt{B}G) \Big) \mathfrak{q}(p-1,p)\Big]\nonumber\\
&=\mathbb{E}\Big[\Big(\frac{1}{N^2}\sum_{i=1}^N \sum_{k:k\neq i} \big(d_i\ntr (\wt{B}^{\la i\ra}G^{\la i\ra}) \big) \frac{\partial (\mathbf{e}_k^*G\mathbf{e}_i)}{\partial g_{ik}} \Big) \mathfrak{q}(p-1,p)\Big]\nonumber\\
&\qquad\qquad-\mathbb{E}\Big[\Big(\frac{1}{N}\sum_{i=1}^N\big(d_i \ntr (\wt{B}^{\la i\ra}G^{\la i\ra})\big)\mathring{T}_i\Big) \mathfrak{q}(p-1,p)\Big]+\mathbb{E}\big[O_\prec(\Psi^2)\mathfrak{q}(p-1,p)\big], \label{021660}
\end{align}
For brevity, for each $i\in \llbracket 1, N\rrbracket$, we set
\begin{align*}
d_{i,2}\equiv d_{i,2}(z)\deq d_i\ntr (\wt{B}^{\la i\ra}G^{\la i\ra})\, , 
\end{align*}
which is independent of $\mathbf{g}_i$. 
Recall the definition of $\mathring{T}_i$ in~(\ref{022030}). Using the integration by parts formula~(\ref{030420}) for the second term on the right side of~(\ref{021660}),  we have
\begin{align}
\mathbb{E}\Big[\Big(\frac{1}{N}\sum_{i=1}^Nd_{i,2} \mathring{T}_i  \Big) \mathfrak{q}(p-1,p)\Big]
&=\mathbb{E}\Big[\Big(\frac{1}{N}\sum_{i=1}^N\sum_{k:k\neq i} d_{i,2}\bar{g}_{ik} \mathbf{e}_k^*G\mathbf{e}_i \Big) \mathfrak{q}(p-1,p)\Big]\nonumber\\
&=\mathbb{E}\Big[\Big(\frac{1}{N^2}\sum_{i=1}^N\sum_{k:k\neq i} d_{i,2}\frac{\partial (\mathbf{e}_k^*G\mathbf{e}_i)}{\partial g_{ik}} \Big) \mathfrak{q}(p-1,p)\Big]\nonumber\\
&\qquad+\mathbb{E}\Big[\Big(\frac{p-1}{N^3}\sum_{i=1}^N\sum_{k:k\neq i} d_{i,2}\mathbf{e}_k^*G\mathbf{e}_i  \sum_{j=1}^N d_j\frac{\partial Z_j}{\partial g_{ik}}\Big) \mathfrak{q}(p-2,p)\Big]\nonumber\\
&\qquad+\mathbb{E}\Big[\Big(\frac{p}{N^3}\sum_{i=1}^N\sum_{k:k\neq i} d_{i,2}\mathbf{e}_k^*G\mathbf{e}_i  \sum_{j=1}^N \overline{d_j}\frac{\partial \overline{Z_j}}{\partial g_{ik}}\Big) \mathfrak{q}(p-1,p-1)\Big], \label{021750}
\end{align}
where the first term on the right side cancels the first term on the right side of~(\ref{021660}).  Hence, for~(\ref{part 1}), it suffices to estimate the second term to the fifth term of the right side of ~(\ref{022035}) and the last two terms of~(\ref{021750}).  
  Note that the fourth and fifth terms of the right side of ~(\ref{022035}) have a very similar form as the last two terms in~(\ref{021750}), respectively. 
 In addition, for the second term on the right side of~(\ref{022035}) we use
 \begin{align*}
\frac{\|\mathbf{g}_i\|_2^2-1}{\|\mathbf{g}_i\|_2}=\mathring{\mathbf{g}}_i^*\mathbf{g}_i-1+O_\prec\Big(\frac{1}{N}\Big).
 \end{align*}
Moreover, we can replace $\ntr \wt{B}G$ by $\ntr \wt{B}^{\la i\ra}G^{\la i\ra}$ in  the second term on the right side of~(\ref{022035}), up to an error $O_\prec(\Psi^2)$, according to Corollary~\ref{cor. finite perturbation}.  Let $Q_i=I$ or $\wt{B}^{\la i\ra}$. In addition, we use the notation $\mathring{Q}_i$ to denote the matrix obtained from $Q_i$ via replacing its $(i,i)$-th entry by zero. Choosing $Q_i=I$, we see that for the second term on the right side of~(\ref{022035}) is of the form 
 \begin{align}
 \mathbb{E}\Big[\Big(\frac{1}{N}\sum_{i=1}^N \wt{d}_i\big(\mathring{\mathbf{g}}_i^*Q_i \mathbf{g}_i-\ntr \mathring{Q}_i\big) G_{ii} \Big) \mathfrak{q}(p-1,p)\Big]+\mathbb{E}\Big[O_\prec(\Psi^2) \mathfrak{q}(p-1,p)\Big], \label{032201}
 \end{align}
 for some $\mathbf{g}_i$-independent quantities $\wt{d}_i\equiv \wt{d}_i(z)$ satisfying $|\wt{d}_i(z)|\prec 1$ uniformly on  $\mathcal{S}_{\mathcal{I}}(\eta_{\mathrm{m}},1)$ and in $i\in\llbracket1, N\rrbracket$. 
Now, using Lemma~\ref{lem.022302} below to estimate third term to the fifth term of the right side of~(\ref{022035}) and the last two terms of~(\ref{021750}), and using  Lemma ~\ref{lem.022303} below to estimate the second term on the right side of~(\ref{022035}), we can conclude the proof of ~(\ref{part 1}).

To prove~(\ref{part 2}), we use the approximation
\begin{align*}
\mathring{\mathbf{h}}_i^*\wt{B}^{\la i\ra}\mathbf{h}_i=\mathring{\mathbf{g}}_i^*\wt{B}^{\la i\ra}\mathbf{g}_i+O_\prec\Big(\frac{1}{N}\Big).
\end{align*}
Moreover, we can replace $\ntr G$ by $\ntr G^{\la i\ra}$ in the left side~(\ref{part 2}), up to any error $O_\prec(\Psi^2)$, according to Corollary~\ref{cor. finite perturbation}. Similarly, choosing $Q_i=\wt{B}^{\la i\ra}$, we see that  the left side of~(\ref{part 2}) is of the form~(\ref{032201}), in light of the fact $\ntr \wt{B}^{\la i\ra}=\ntr B=0$. Hence,~(\ref{part 2}) follows from Lemma ~\ref{lem.022303} below directly. This completes the proof of Lemma~\ref{lem. 022301}.
\end{proof}
It remains is to prove the following two lemmas.
\begin{lem}\label{lem.022302} Suppose that the assumptions in Theorem~\ref{thm.convergence rate} and~\eqref{le b equation} hold. Letting $\widehat{d}_1, \cdots, \widehat{d}_N\in \mathbb{C}$ be any possibly $z$-dependent random variables satisfying $\max_{i\in\llbracket 1, N\rrbracket}|\widehat{d}_i|\prec 1$ uniformly on $\mathcal{S}_{\mathcal{I}}(\eta_{\mathrm{m}},1)$, and letting $Q_i=I$ or $\wt{B}^{\la i\ra}$, we have the estimates
\begin{align}
&\frac{1}{N^2} \sum_{i=1}^N \sum_{k:k\neq i} \widehat{d}_i \frac{\partial \|\mathbf{g}_i\|_2^{-1}}{\partial g_{ik}} \mathbf{e}_k^* \wt{B}^{\la i\ra}G\mathbf{e}_i=O_\prec(\Psi^2),\label{022350}\\
&\frac{1}{N^3}\sum_{i=1}^N\sum_{k:k\neq i} \widehat{d}_i\mathbf{e}_k^* Q_iG\mathbf{e}_i \sum_{j=1}^N d_j\frac{\partial Z_j}{\partial g_{ik}}=O_\prec (\Psi^4),\label{022351}
\end{align}  
and the same estimates hold if we replace $d_j$ and $Z_j$ by their complex conjugates in~(\ref{022351}).
\end{lem}

\begin{lem}\label{lem.022303}
Suppose that the assumptions in Theorem~\ref{thm.convergence rate} and~\eqref{le b equation} hold. Let $\wt{d}_1, \ldots, \wt{d}_N\in \mathbb{C}$ be any possibly $z$-dependent random variables satisfying $\max_{i\in \llbracket 1, N\rrbracket}|\wt{d}_i|\prec 1$ uniformly on $\mathcal{S}_{\mathcal{I}}(\eta_{\mathrm{m}},1)$. Assume that $\wt{d}_i$ is independent of $\mathbf{g}_i$ for each $i\in \llbracket1, N\rrbracket$. Let $Q_i=I$ or $\wt{B}^{\la i\ra}$. We have the estimate
\begin{align*}
 &\mathbb{E}\Big[\Big(\frac{1}{N}\sum_{i=1}^N \wt{d}_i\big(\mathring{\mathbf{g}}_i^*Q_i \mathbf{g}_i-\ntr \mathring{Q}_i\big) G_{ii} \Big) \mathfrak{q}(p-1,p)\Big]\nonumber\\
 &\qquad\qquad=\mathbb{E}\Big[O_\prec(\Psi^2)\mathfrak{q}(p-1,p)\Big]+\mathbb{E}\Big[O_\prec(\Psi^4)\mathfrak{q}(p-2,p)\Big]+\mathbb{E}\Big[O_\prec(\Psi^4)\mathfrak{q}(p-1,p-1)\Big].
 \end{align*}
\end{lem}

\begin{proof}[Proof of Lemma~\ref{lem.022302}]
For~(\ref{022350}), we have
\begin{align*}
\frac{1}{N^2} \sum_{i=1}^N \sum_{k:k\neq i} \widehat{d}_i \frac{\partial \|\mathbf{g}_i\|_2^{-1}}{\partial g_{ik}} \mathbf{e}_k^* \wt{B}^{\la i\ra}G\mathbf{e}_i&=-\frac{1}{2N^2} \sum_{i=1}^N \sum_{k:k\neq i} \frac{\widehat{d}_i }{\|\mathbf{g}_i\|_2^3} \bar{g}_{ik} \mathbf{e}_k^* \wt{B}^{\la i\ra}G\mathbf{e}_i\nonumber\\
&=-\frac{1}{2N^2} \sum_{i=1}^N \frac{\widehat{d}_i }{\|\mathbf{g}_i\|_2^3}  \mathring{\mathbf{g}}_i^* \wt{B}^{\la i\ra}G\mathbf{e}_i=O_\prec(\frac{1}{N})=O_\prec(\Psi^2),
\end{align*}
where in the third step we used~(\ref{030460}).

In the sequel, we prove~(\ref{022351}). By the definition of $Z_j$ in~(\ref{def of Z})  and $\Upsilon$ in~(\ref{definition of Upsilon}), and the identities in~(\ref{021701}), we can write
\begin{align*}
Z_j=&(\wt{B}G)_{jj}\ntr G-G_{jj}(\ntr (\wt{B}G)-\Upsilon)=(\wt{B}G)_{jj}\ntr G-G_{jj}\Big(\big(\ntr (\wt{B}G)\big)^2-\ntr G \ntr (\wt{B}G\wt{B})\Big)\nonumber\\
=& \ntr G-G_{jj}\bigg( \Big(1-\ntr \big((A-z) G\big)\Big)^2-\ntr G\Big( \ntr \big((A-z)^2 G\big)-a_j+2z\Big)\bigg)
\end{align*}
Hence, we have
\begin{align}
\frac{\partial Z_j}{\partial g_{ik}}=&\ntr \Big(\frac{\partial G}{\partial g_{ik}}\Big)-\mathbf{e}_j^*\frac{\partial G}{\partial g_{ik}}\mathbf{e}_j\big( \mathcal{A}_1+a_j\ntr G\big)+\mathcal{A}_2G_{jj}\ntr \Big((A-z)\frac{\partial G}{\partial g_{ik}}\Big)\nonumber\\
&+G_{jj}\ntr \Big(\frac{\partial G}{\partial g_{ik}}\Big)\big( \mathcal{A}_3-a_j\big)+G_{jj} \ntr G \ntr\Big((A-z)^2\frac{\partial G}{\partial g_{ik}}\Big), \label{022245}
\end{align}
where we introduced the shorthand notations 
\begin{align*}
\mathcal{A}_1\equiv \mathcal{A}_1(z)&\deq \big(1-\ntr \big(A-z) G\big)\big)^2-\ntr G\big( \ntr \big((A-z)^2 G\big)+2z\big),\nonumber\\
\mathcal{A}_2\equiv \mathcal{A}_2(z)&\deq 2\big(1-\ntr \big((A-z) G\big)\big),\nonumber\\  
\mathcal{A}_3\equiv \mathcal{A}_3(z)&\deq \ntr \big((A-z)^2 G\big)+2z
\end{align*}
to denote some $O_\prec (1)$ tracial quantities whose explicit formulas are irrelevant for our analysis below.  
In addition, recalling the notation $\Delta_G(i,k)$ from~(\ref{022360}), we denote
\begin{align}
\Delta_{Z_j}(i,k)\deq &\ntr \big(\Delta_G(i,k)\big)-\mathbf{e}_j^*\Delta_G(i,k)\mathbf{e}_j\big(\mathcal{A}_1+a_j\ntr G\big)+\mathcal{A}_2G_{jj}\ntr \big((A-z)\Delta_G(i,k)\big)\nonumber\\
&+G_{jj}\ntr \big(\Delta_G(i,k)\big)\big( \mathcal{A}_3-a_j\big)+G_{jj} \ntr G \ntr\big((A-z)^2\Delta_G(i,k)\big). \label{022250}
\end{align}

For convenience, we introduce the matrix 
\begin{align}
D\deq \text{diag} (d_i), \label{022251}
\end{align}
and the shorthand notation
\begin{align}
\mathbf{w}_i\deq c_i(\mathbf{e}_i+\mathbf{h}_i), \label{022252}
\end{align}
where $c_i$ is defined in~(\ref{022140}).

Substituting~(\ref{022240}) into~(\ref{022245}) and using the notations defined in~(\ref{022250})-(\ref{022252}), we obtain
\begin{align}
\sum_{j=1}^N d_j\frac{\partial Z_j}{\partial g_{ik}}&=  \Big(\mathbf{w}_i^*\wt{B}^{\la i\ra} R_i G^2\mathbf{e}_k + \mathbf{w}_i^* G^2 R_i \wt{B}^{\la i\ra} \mathbf{e}_k\Big) \Big(\ntr D+\ntr \big(DG\big)\mathcal{A}_3-\ntr \big(ADG\big)\Big)\nonumber\\
&\qquad-\Big( \mathbf{w}_i^* \wt{B}^{\la i\ra} R_i GDG\mathbf{e}_k+\mathbf{w}_i^* GDGR_i\wt{B}^{\la i\ra} \mathbf{e}_k \Big) \mathcal{A}_1\nonumber\\
&\qquad-\Big(\mathbf{w}_i^* \wt{B}^{\la i\ra} R_i GADG\mathbf{e}_k +\mathbf{w}_i^* GADGR_i\wt{B}^{\la i\ra} \mathbf{e}_k \Big)\ntr G\nonumber\\
&\qquad+\Big( \mathbf{w}_i^*\wt{B}^{\la i\ra} R_i G(A-z)G\mathbf{e}_k+ \mathbf{w}_i^* G(A-z)G R_i \wt{B}^{\la i\ra} \mathbf{e}_k\Big)\ntr \big(DG\big)\mathcal{A}_2\nonumber\\
&\qquad+\Big( \mathbf{w}_i^*\wt{B}^{\la i\ra} R_i G(A-z)^2G\mathbf{e}_k+ \mathbf{w}_i^* G(A-z)^2G R_i \wt{B}^{\la i\ra} \mathbf{e}_k\Big)\ntr G \ntr \big(DG\big)\nonumber\\
&\qquad+\sum_{j=1}^N d_j \Delta_{Z_j}(i,k). \label{0217105}
\end{align}
Since $|G_{ii}|\prec 1$ for all $i\in \llbracket 1, N\rrbracket$ (\cf~(\ref{030321})), we have $|\ntr (Q G)|\prec 1$ for all diagonal matrix satisfying $\|Q\|\prec 1$. Therefore, except for the last term, all the other terms  in~(\ref{0217105}) are of the form
\begin{align*}
\widehat{d}\Big(\mathbf{w}_i^*\wt{B}^{\la i\ra} R_i GQG\mathbf{e}_k + \mathbf{w}_i^* GQG R_i \wt{B}^{\la i\ra} \mathbf{e}_k\Big)
\end{align*}
for some $z$-dependent quantity $\widehat{d}\equiv \widehat{d}(z)$ satisfying $|\widehat{d}|\prec 1$ uniformly on $\mathcal{S}_{\mathcal{I}}(\eta_{\mathrm{m}},1)$, and some diagonal matrix $Q$ with $\|Q\|\lesssim 1$, which can be $I$, $D$, $AD$, $A-z$ or $(A-z)^2$.  Hence, to establish~(\ref{022351}), it suffices to estimate 
\begin{align}
\frac{1}{N^3}\sum_{i=1}^N\sum_{k:k\neq i}\widehat{d}_i\mathbf{e}_k^* Q_iG\mathbf{e}_i \Big(\mathbf{w}_i^*\wt{B}^{\la i\ra} R_i GQG\mathbf{e}_k + \mathbf{w}_i^* GQG R_i \wt{B}^{\la i\ra} \mathbf{e}_k\Big) \label{0223100}
\end{align}
and 
\begin{align}
\frac{1}{N^3}\sum_{i=1}^N\sum_{k:k\neq i}\widehat{d}_i \mathbf{e}_k^* Q_iG\mathbf{e}_i \sum_{j=1}^N d_j \Delta_{Z_j}(i,k)\label{0223101}
\end{align}
for any possibly $z$-dependent random variables $\widehat{d}_1, \cdots, \widehat{d}_N\in \mathbb{C}$ which satisfy $\max_{i\in\llbracket 1, N\rrbracket}|\widehat{d}_i|\prec 1$ uniformly on $\mathcal{S}_{\mathcal{I}}(\eta_{\mathrm{m}},1)$.

The following lemma provides the bound on the quantity in~(\ref{0223101}).
\begin{lem} \label{lem.022405}Suppose that the assumptions in Theorem~\ref{thm.convergence rate} and~\eqref{le b equation} hold. Letting $\widehat{d}_1, \ldots, \widehat{d}_N\in \mathbb{C}$ be any possibly $z$-dependent random variables satisfying $\max_{i\in \llbracket 1, N\rrbracket}|\widehat{d}_i|\prec 1$ uniformly on $\mathcal{S}_{\mathcal{I}}(\eta_{\mathrm{m}},1)$, and letting $Q_i=I$ or $\wt{B}^{\la i\ra}$, we have 
\begin{align}
\frac{1}{N^3}\sum_{i=1}^N\sum_{k:k\neq i}\widehat{d}_i \mathbf{e}_k^* Q_iG\mathbf{e}_i \sum_{j=1}^N d_j \Delta_{Z_j}(i,k)=O_\prec(\Psi^4) \label{022461}
\end{align}
uniformly on $\mathcal{S}_{\mathcal{I}}(\eta_{\mathrm{m}},1)$.
\end{lem} 
The proof of Lemma~\ref{lem.022405} will also be postponed to Appendix~\ref{a.A}.

With Lemma~\ref{lem.022405}, it suffices to estimate~(\ref{0223100}) below. Note that 
\begin{align}
&\mkern-18mu\mkern-18mu\frac{1}{N^3}\sum_{i=1}^N\sum_{k:k\neq i}\widehat{d}_i\mathbf{e}_k^* Q_iG\mathbf{e}_i \Big(\mathbf{w}_i^*\wt{B}^{\la i\ra} R_i GQG\mathbf{e}_k + \mathbf{w}_i^* GQG R_i \wt{B}^{\la i\ra} \mathbf{e}_k\Big)\nonumber\\
&=\frac{1}{N^3}\sum_{i=1}^N\widehat{d}_i\Big(\mathbf{w}_i^*\wt{B}^{\la i\ra} R_i GQGQ_iG\mathbf{e}_i  + \mathbf{w}_i^* GQG R_i \wt{B}^{\la i\ra}Q_iG\mathbf{e}_i \Big)\nonumber\\
&\qquad-\frac{1}{N^3}\sum_{i=1}^N\widehat{d}_i\mathbf{e}_i^* Q_iG\mathbf{e}_i \Big(\mathbf{w}_i^*\wt{B}^{\la i\ra} R_i GQG\mathbf{e}_i + \mathbf{w}_i^* GQG R_i \wt{B}^{\la i\ra} \mathbf{e}_i\Big). \label{0217103}
\end{align}
Recall the fact that $Q_i=I$ or $\wt{B}^{\la i\ra}$. Now, using the facts $\wt{B}^{\la i\ra}=R_i \wt{B} R_i$ and $R_i^2=I$, we have the following relations
\begin{align}
&\wt{B}^{\la i\ra}=\wt{B} -\mathbf{r}_i\mathbf{r}_i^*\wt{B}-\wt{B} \mathbf{r}_i\mathbf{r}_i^*+\mathbf{r}_i\mathbf{r}_i^*\wt{B} \mathbf{r}_i\mathbf{r}_i^*,\qquad \wt{B}^{\la i\ra} R_i=\wt{B}-\mathbf{r}_i\mathbf{r}_i^*\wt{B},\nonumber\\
& R_i\wt{B}^{\la i\ra}=\wt{B}- \wt{B}\mathbf{r}_i\mathbf{r}_i^*,\qquad R_i(\wt{B}^{\la i\ra})^2=\wt{B}^2-\wt{B}^2\mathbf{r}_i\mathbf{r}_i^*,
\label{0223110}
\end{align}
\ie the $i$-dependence of these quantities are shifted to $\mathbf{r}_i$.
Recalling the notations $\mathbf{w}_i=c_i(\mathbf{e}_i+\mathbf{h}_i)$ and $\mathbf{r}_i=\ell_i(\mathbf{e}_i+\mathbf{h}_i)$, 
and using~(\ref{030430}) and~(\ref{0223110}) to~(\ref{0217103}) for either $Q_i=I$ or $\wt{B}^{\la i\ra}$, it is not difficult to check that the right side of~(\ref{0217103}) is the sum of terms in the form
\begin{align}
&\frac{1}{N^3}\sum_{i=1}^Ny_i \mathbf{e}_i^* (\star \; G \star G \star G)\mathbf{e}_i,& &\frac{1}{N^3}\sum_{i=1}^Ny_i \mathbf{h}_i^* (\star\; G\star G\star G)\mathbf{e}_i,\label{0223125}\\
& \frac{1}{N^3}\sum_{i=1}^Ny_i \mathbf{e}_i^* (\star\; G\star G\; \star )\boldsymbol{\alpha}_i\boldsymbol{\beta}_i^*Q_i'G\mathbf{e}_i,& &\frac{1}{N^3}\sum_{i=1}^Ny_i \mathbf{h}_i^* (\star\; G\star G\;\star )\boldsymbol{\alpha}_i\boldsymbol{\beta}_i^*Q_i'G\mathbf{e}_i, \label{0223120}
\end{align}
where $y_1, \ldots, y_N\in \mathbb{C}$ are some random variables (which can be different from line to line) satisfying $\max_{i\in \llbracket 1, N\rrbracket}|y_i|\prec 1$, and where each $\star$ either stands for one of the matrices  $I$, $A$, $A-z$,  $\wt{B}$, $D$ or the product of some of them (which can be different from one to another), but are all $i$-independent and their operator norms are $O_\prec(1)$. In addition, $\boldsymbol{\alpha}_i, \boldsymbol{\beta}_i=\mathbf{e}_i$ or $\mathbf{h}_i$ and $Q_i'=I$, $\wt{B}^{\la i\ra}$ or $\wt{B}$ in~(\ref{0223120}). 

Now, recall the fact that $\mathbf{v}_i$ is the $i$-th column of $U$, \ie $U\mathbf{e}_i=\mathbf{v}_i$, and $\mathbf{h}_i=\e{-\mathrm{i}\theta_i}\mathbf{v}_i$ from~(\ref{some notation for r}). Therefore we have the following identities: for any diagonal matrix $Y\deq \text{diag}(y_i)$,
 \begin{align}
 \hspace{-2ex}\sum_{i=1}^N y_i \mathbf{e}_i\mathbf{e}_i^* =Y,\quad \sum_{i=1}^N y_i \mathbf{h}_i \mathbf{h}_i^*=UY U^*, \quad \sum_{i=1}^N y_i \mathbf{h}_i\mathbf{e}_i^* = U Y\Theta^*, \quad \sum_{i=1}^N y_i \mathbf{e}_i\mathbf{h}_i^* =  Y\Theta U^*, \label{02231251}
 \end{align}
 where $\Theta\deq \text{diag}(\e{\mathrm{i}\theta_i})$.  
Applying~(\ref{02231251}) to the quantities in~(\ref{0223125}), and using  $\|G(z)\|\leq \eta^{-1}$, we get
\begin{align*}
&\frac{1}{N^3}\sum_{i=1}^Ny_i \mathbf{e}_i^* \big(\star G \star G \star G\big)\mathbf{e}_i=\frac{1}{N^2}\ntr \big(\star G\star G\star GY\big)=O_\prec\Big(\frac{\ntr |G|^2}{N^2\eta}\Big)=O_\prec\Big(\frac{\Im \ntr G }{N^2\eta^2}\Big)=O(\Psi^4), \nonumber\\
& \frac{1}{N^3}\sum_{i=1}^Ny_i \mathbf{h}_i^* \big(\star G \star G \star G\big)\mathbf{e}_i= \frac{1}{N^2} \ntr (\star G \star G \star G Y\Theta U^*)=O_\prec\Big(\frac{\ntr |G|^2}{N^2\eta}\Big)=O(\Psi^4).
\end{align*}
For the terms in~(\ref{0223120}), we set $\widehat{y}_i=y_i \boldsymbol{\beta}_i^*Q_i'G\mathbf{e}_i$ and $\widehat{Y}\deq \text{diag}(\widehat{y}_i)$. First, we claim $|\widehat{y}_i|\prec 1$. Such a bound follows from~(\ref{030321}) in case $Q_i'=I$ or $\wt{B}^{\la i\ra}$. In case of $\boldsymbol{\beta}_i=\mathbf{e}_i$ and $Q_i'=\wt{B}$,  we have $\widehat{y}_i=y_i(\wt{B}G)_{ii}$ which is $O_\prec(1)$, according to $(\wt{B}G)_{ii}=1-(a_i-z)G_{ii}$ and $|G_{ii}|\prec 1$; in case of $\boldsymbol{\beta}_i=\mathbf{h}_i$ and $Q_i'=\wt{B}$, we can use~(\ref{030430}) to get $\widehat{y}_i=y_i\mathbf{h}_i^*\wt{B}G\mathbf{e}_i=-y_ib_i\mathbf{h}_i^*R_iG\mathbf{e}_i=y_ib_iG_{ii}$ and thus $|\widehat{y}_i|\prec 1$. Consequently, we have $\|\widehat{Y}\|\prec 1$. Applying this fact together with~(\ref{02231251}) to the quantities in~(\ref{0223120}), we obtain 
\begin{align*}
 \frac{1}{N^3}\sum_{i=1}^Ny_i \mathbf{e}_i^* \big(\star G \star G\star \big)\boldsymbol{\alpha}_i \boldsymbol{\beta}_i^*Q_i'G\mathbf{e}_i&= \frac{1}{N^3}\sum_{i=1}^N\widehat{y}_i \mathbf{e}_i^* \big(\star G \star G\star\big)\boldsymbol{\alpha}_i\nonumber\\
&=\frac{1}{N^2} \ntr \Big(\star G\star G\star \big(\sum_{i=1}^N \widehat{y}_i \boldsymbol{\alpha}_i\mathbf{e}_i^*\big)\Big)=O_\prec\Big(\frac{\Im \ntr G }{N^2\eta}\Big)=O_\prec(\Psi^4),\nonumber\\
\end{align*}
and
\begin{align*}
 \frac{1}{N^3}\sum_{i=1}^Ny_i \mathbf{h}_i^* \big(\star G\star G\star \big)\boldsymbol{\alpha}_i \boldsymbol{\beta}_i^*Q_i'G\mathbf{e}_i&=\frac{1}{N^3}\sum_{i=1}^N\widehat{y}_i \mathbf{h}_i^* \big(\star G\star G\star \big)\boldsymbol{\alpha}_i \nonumber\\
&=\frac{1}{N^2} \ntr \Big(\star G\star G\star \big(\sum_{i=1}^N\widehat{y}_i \boldsymbol{\alpha}_i\mathbf{h}_i^*\big)\Big)=O_\prec\Big(\frac{\Im \ntr G }{N^2\eta}\Big)=O_\prec(\Psi^4),
\end{align*}
where we used the fact $\boldsymbol{\alpha}_i=\mathbf{e}_i$ or $\mathbf{h}_i$ and the identities in~(\ref{02231251}) to show $\|\sum_{i=1}^N\widehat{y}_i \boldsymbol{\alpha}_i\boldsymbol{\beta}_i^*\|\prec1$ for $\boldsymbol{\alpha}_i, \boldsymbol{\beta}_i=\mathbf{e}_i$ or $\mathbf{h}_i$.
Hence, we conclude the proof of Lemma~\ref{lem.022302}.
 \end{proof}

\begin{proof}[Proof of Lemma~\ref{lem.022303}] By assumption, both $\wt{d}_i$ and $Q_i$ are independent of $\mathbf{g}_i$ for each $i\in \llbracket 1, N\rrbracket$. Hence, using integration by parts formula~(\ref{030420}), we obtain
\begin{align}
&\mkern-18mu\mathbb{E}\Big[\Big(\frac{1}{N}\sum_{i=1}^N \wt{d}_i\big(\mathring{\mathbf{g}}_i^*Q_i \mathbf{g}_i-\ntr \mathring{Q}_i\big) G_{ii} \Big) \mathfrak{q}(p-1,p)\Big]\nonumber\\
&\qquad=\mathbb{E}\Big[\Big(\frac{1}{N}\sum_{i=1}^N \wt{d}_i\sum_{k:k\neq i}\bar{g}_{ik}\mathbf{e}_k^*Q_i \mathbf{g}_i G_{ii} \Big) \mathfrak{q}(p-1,p)\Big]-\mathbb{E}\Big[\Big(\frac{1}{N}\sum_{i=1}^N \wt{d}_i \ntr \mathring{Q}_i  G_{ii} \Big) \mathfrak{q}(p-1,p)\Big]\nonumber\\
&\qquad=\mathbb{E}\Big[\Big(\frac{1}{N^2}\sum_{i=1}^N \wt{d}_i\sum_{k:k\neq i}\mathbf{e}_k^*Q_i \mathbf{g}_i \mathbf{e}_i^*\frac{\partial G}{\partial g_{ik}}\mathbf{e}_i \Big) \mathfrak{q}(p-1,p)\Big]\nonumber\\
&\qquad\qquad+\mathbb{E}\Big[\Big(\frac{p-1}{N^2}\sum_{i=1}^N \wt{d}_i\sum_{k:k\neq i}\mathbf{e}_k^*Q_i \mathbf{g}_i G_{ii} \frac{1}{N}\sum_{j=1}^N d_j\frac{\partial Z_j}{\partial g_{ik}}\Big) \mathfrak{q}(p-2,p)\Big]\nonumber\\
&\qquad\qquad+\mathbb{E}\Big[\Big(\frac{p}{N^2}\sum_{i=1}^N \wt{d}_i\sum_{k:\neq i}\mathbf{e}_k^*Q_i \mathbf{g}_i G_{ii} \frac{1}{N}\sum_{j=1}^N\overline{d_j}\frac{\overline{\partial Z_j}}{\partial g_{ik}}\Big) \mathfrak{q}(p-1,p-1)\Big]. \label{0223200}
\end{align}
We start with the first term of the right side of~(\ref{0223200}). Recalling~(\ref{022140}) and ~(\ref{022240}), and using the shorthand notation $\widehat{d}_i\deq \wt{d}_i  c_i$, we have
\begin{align*}
\frac{1}{N^2}\sum_{i=1}^N &\wt{d}_i\sum_{k:k\neq i}\mathbf{e}_k^*Q_i \mathbf{g}_i \mathbf{e}_i^*\frac{\partial G}{\partial g_{ik}}\mathbf{e}_i\nonumber\\
=& \frac{1}{N^2}\sum_{i=1}^N \widehat{d}_i \sum_{k:k\neq i}\mathbf{e}_k^*Q_i \mathbf{g}_i \Big[\mathbf{e}_i^*G\mathbf{e}_k (\mathbf{e}_i+\mathbf{h}_i)^*\wt{B}^{\la i\ra} R_i G\mathbf{e}_i+\mathbf{e}_i^*GR_i\wt{B}^{\la i\ra}\mathbf{e}_k (\mathbf{e}_i+\mathbf{h}_i)^*G\mathbf{e}_i\Big]\nonumber\\
&\qquad+ \frac{1}{N^2}\sum_{i=1}^N \widehat{d}_i\sum_{k:k\neq i}\mathbf{e}_k^*Q_i \mathbf{g}_i  \mathbf{e}_i^*\Delta_G(i,k)\mathbf{e}_i\nonumber\\
=& \frac{1}{N^2}\sum_{i=1}^N \widehat{d}_i  \Big[-\mathbf{e}_i^*GQ_i \mathbf{g}_i\big(b_iT_i+(\wt{B}G)_{ii}\big)+\mathbf{e}_i^*GR_i\wt{B}^{\la i\ra}Q_i \mathbf{g}_i \big(G_{ii}+T_i\big)\Big]\nonumber\\
&\qquad-\frac{1}{N^2}\sum_{i=1}^N \widehat{d}_i \mathbf{e}_i^*Q_i \mathbf{g}_i \Big[-G_{ii} \big(b_iT_i+(\wt{B}G)_{ii}\big)-b_i\mathbf{e}_i^*G\mathbf{h}_i\big(G_{ii}+T_i\big)\Big]\nonumber\\
&\qquad+ \frac{1}{N^2}\sum_{i=1}^N \widehat{d}_i\sum_{k:k\neq i}\mathbf{e}_k^*Q_i \mathbf{g}_i  \mathbf{e}_i^*\Delta_G(i,k)\mathbf{e}_i\nonumber\\
=& \frac{1}{N^2}\sum_{i=1}^N \widehat{d}_i\sum_{k:k\neq i}\mathbf{e}_k^*Q_i \mathbf{g}_i  \mathbf{e}_i^*\Delta_G(i,k)\mathbf{e}_i+O_\prec(\Psi^2),
\end{align*}
where in the second step we separated the sum $\sum_i\sum_{k:k\neq i}= \sum_{k,i}-\sum_{k=i}$ and used~(\ref{022650}), and in the last step we used the bound~(\ref{030321}) again. 
Then the estimate 
\begin{align}
\frac{1}{N^2}\sum_{i=1}^N \wt{d}_i\sum_{k:k\neq i}\mathbf{e}_k^*Q_i \mathbf{g}_i \mathbf{e}_i^*\frac{\partial G}{\partial g_{ik}}\mathbf{e}_i=O_\prec (\Psi^2) \label{022653}
\end{align}
is implied by the following lemma, whose proof will be postponed to Appendix~\ref{a.A}. 

\begin{lem} \label{lem.022407}Suppose that the assumptions of Theorem~\ref{thm.convergence rate} and~\eqref{le b equation} hold. Let $\widehat{d}_1, \cdots, \widehat{d}_N$ be any possibly $z$-dependent complex random variables satisfying ${\max_{i\in\llbracket 1, N\rrbracket}}|\widehat{d}_i|\prec 1$ uniformly on $\mathcal{S}_{\mathcal{I}}(\eta_{\mathrm{m}},1)$, and let $Q_i=I$ or~$\wt{B}^{\la i\ra}$. Then,
\begin{align}
\frac{1}{N^2}\sum_{i=1}^N \widehat{d}_i\sum_{k:k\neq i}\mathbf{e}_k^*Q_i \mathbf{g}_i  \mathbf{e}_i^*\Delta_G(i,k)\mathbf{e}_i=O_\prec(\Psi^2). \label{022470}
\end{align}
\end{lem}
 Now we investigate the last two terms of~(\ref{0223200}). Let $\widehat{d}_1, \ldots,\widehat{d}_N$ be any possibly $z$-dependent complex random variables satisfying $|\widehat{d}_i|\prec 1$ uniformly on $\mathcal{S}_{\mathcal{I}}(\eta_{\mathrm{m}},1)$.  Let $Q_i=I$ or $\wt{B}^{\la i\ra}$. We claim that 
\begin{align}
\frac{1}{N^3}\sum_{i=1}^N\sum_{k:k\neq i} \widehat{d}_i\mathbf{e}_k^* Q_i\mathbf{g}_i \sum_{j=1}^N d_j\frac{\partial Z_j}{\partial g_{ik}}=O_\prec (\Psi^4) \label{022651}
\end{align}
holds uniformly on $\mathcal{S}_{\mathcal{I}}(\eta_{\mathrm{m}},1)$, and the same estimate holds if we replace $d_j$ and $Z_j$ by their complex conjugates. 
The proof of~(\ref{022651}) is nearly the same as~(\ref{022351}).  The only difference is a missing $G$ in the factor $\mathbf{e}_k^* Q_i\mathbf{g}_i$ which played no essential r\^ole in the proof of~(\ref{022351}). We omit the details of the proof of~(\ref{022651}). 	

Using~(\ref{022653}) and~(\ref{022651}) to~(\ref{0223200}), we can conclude the proof of Lemma~\ref{lem.022303}.
\end{proof}

\appendix

\section{}\label{the appendix A}
In this appendix, we collect some basic tools from random matrix theory.
\subsection{Stochastic domination and large deviation properties}\label{stochastic domination section}
Recall the stochastic domination in Definition~\ref{definition of stochastic domination}. The relation $\prec$ is a partial ordering: it is transitive and it satisfies the arithmetic rules of an order relation, {\it e.g.}, if $X_1\prec Y_1$ and $X_2\prec Y_2$ then $X_1+X_2\prec Y_1+Y_2$ and $X_1 X_2\prec Y_1 Y_2$. Further assume that $\Phi(v)\ge N^{-C}$ is deterministic and that~$Y(v)$ is a nonnegative random variable satisfying $\E [Y(v)]^2\le N^{C'}$ for all~$v$. Then $Y(v) \prec \Phi(v)$, uniformly in $v$, implies $\E [Y(v)] \prec \Phi(v)$, uniformly in~$v$.

Gaussian vectors have well-known large deviation properties. We will use them in the following form
whose proof is standard.  
\begin{lem} \label{lem.091720} Let $X=(x_{ij})\in M_N(\C)$ be a deterministic matrix and let $\bs{y}=(y_{i})\in\C^N$ be a deterministic complex vector. For a Gaussian real or complex random vector $\mathbf{g}=(g_1,\ldots, g_N)\in \mathcal{N}_{\mathbb{R}}(0,\sigma^2 I_N)$ or $\mathcal{N}_{\mathbb{C}}(0,\sigma^2 I_N)$, we have
 \begin{align}\label{091731}
  |\bs{y}^* \bs{g}|\prec\sigma \|\bs{y}\|_2,\qquad\qquad  |\bs{g}^* X\bs{g}-\sigma^2N \ntr X|\prec \sigma^2\| X\|_2.
 \end{align}
\end{lem}

\subsection{Rank-one perturbation formula}
At various places, we use the following fundamental perturbation formula: for $\bs{\alpha},\bs{\beta}\in\C^N$ and an invertible $D\in M_N(\C)$, we have
\begin{align}
\big(D+\bs{\alpha}\bs{\beta}^*\big)^{-1}=D^{-1}-\frac{D^{-1}\bs{\alpha}\bs{\beta}^*D^{-1}}{1+\bs{\beta}^*D^{-1}\bs{\alpha}}, \label{091002Kevin}
\end{align}
as can be checked readily. A standard application of~\eqref{091002Kevin} is recorded in the following lemma.
\begin{lem}\label{finite rank perturbation}
Let $D\in M_N(\C)$ be Hermitian and let $Q\in M_N(\C)$ be arbitrary. Then, for any finite-rank Hermitian matrix $R\in M_N(\C)$, we have
\begin{align}
\bigg|\ntr \Big(Q\big(D+R-z\big)^{-1}\Big)-\ntr \Big(Q(D-z)^{-1}\Big)\bigg| &\leq \frac{\mathrm{rank}(R)\|Q\|}{N\eta},\qquad z=E+\ii\eta\in\C^+.
 \label{091002}
\end{align}
\end{lem}
Using Lemma~\ref{finite rank perturbation}, we also have the following corollary. 
\begin{cor} \label{cor. finite perturbation}With the notations in~(\ref{0911401}) and~(\ref{090820}), we have 
\begin{align}
\big|\ntr G-\ntr G^{\la i\ra}\big|\leq C\Psi^2,\quad \big| \ntr \wt{B}^{\la i\ra} G^{\la i\ra}-\wt{B}G\big|\leq C \Psi^2,\quad \big| \ntr \wt{B}^{\la i\ra} G^{\la i\ra}\wt{B}^{\la i\ra}-\wt{B}G\wt{B}\big|\leq C \Psi^2. \label{finite rank perturbation for tracial quantities}
\end{align}
\end{cor}
\begin{proof}
Recalling the Hermitian matrix $H^{\la i\ra}$ defined in~(\ref{090820}), we see that $H$ is a finite rank perturbation of $H^{\la i\ra}$ and the perturbation $H-H^{\la i\ra}$ is obviously Hermitian. Using ~(\ref{091002}) with $Q=I$, $D=H^{\la i\ra}$ and  $R=H-H^{\la i\ra}=B-\wt{B}^{\la i\ra}$, it is straightforward to get the first bound in~(\ref{finite rank perturbation for tracial quantities}). For the second bound, at first, we see that
\begin{align}
\ntr (\wt{B}^{\la i\ra}G)-\ntr (\wt{B}G)&=\ntr (\wt{B}^{\la i\ra}G)-\ntr (R_i\wt{B}^{\la i\ra} R_iG)\nonumber\\
&=\frac{1}{N}\mathbf{r}_i^* \wt{B}^{\la i\ra} G\mathbf{r}_i+\frac{1}{N}\mathbf{r}_i^*G\wt{B}^{\la i\ra}\mathbf{r}_i-\frac{1}{N}\mathbf{r}_i^*\wt{B}^{\la i\ra} \mathbf{r}_i\mathbf{r}_iG\mathbf{r}_i=O_\prec(\frac{1}{N}), \label{030501}
\end{align}
where in the last step we used the fact $\mathbf{r}_i=\ell_i(\mathbf{e}_i+\mathbf{h}_i)$, the estimates in~(\ref{021910}), and the bound in~(\ref{030321}). Then applying ~(\ref{091002}) with $Q=\wt{B}^{\la i\ra}$, $D=H^{\la i\ra}$ and  $R=H-H^{\la i\ra}=B-\wt{B}^{\la i\ra}$, we  obtain 
\begin{align}
\big|\ntr (\wt{B}^{\la i\ra}G^{\la i\ra})-\ntr (\wt{B}^{\la i\ra}G)\big|\leq C\Psi^2. \label{030505}
\end{align}
Combining~(\ref{030501}) and~(\ref{030505}) yields  the second estimate in~(\ref{finite rank perturbation for tracial quantities}). The third one in~(\ref{finite rank perturbation for tracial quantities}) can be verified similarly. We omit the details. So we complete  the proof of Corollary~\ref{cor. finite perturbation}.
\end{proof}

\section{} \label{a.A}
In this appendix, we estimate the terms with $\Delta_R(i,k)$'s involved. More specifically, we will prove Lemmas~\ref{lem.022201},~\ref{lem.022405} and~\ref{lem.022407}.

According to~(\ref{0223300}), we see that $\Delta_R(i,k)$ is the sum of terms of the form 
\begin{align*}
\widehat{d}_i \bar{g}_{ik}\, \boldsymbol{\alpha}_i\boldsymbol{\beta}_i^*,  
\end{align*}
for some $\widehat{d}_i\in \mathbb{C}$ satisfying $|\widehat{d}_i|\prec 1$ uniformly on $\mathcal{S}_{\mathcal{I}}(\eta_{\mathrm{m}},1)$, and $\boldsymbol{\alpha}_i,\boldsymbol{\beta}_i=\mathbf{e}_i$ or $\mathbf{h}_i$. Hereafter $\widehat{d}_i$ can change from line to line, up to the bound $|\widehat{d}_i|\prec 1$ uniformly on $\mathcal{S}_{\mathcal{I}}(\eta_{\mathrm{m}},1)$. Then, by~(\ref{022360}), we see that 
$\Delta_G(i,k)$ is a  sum of the terms of the form 
\begin{align}
\widehat{d}_i \bar{g}_{ik} G\boldsymbol{\alpha}_i\boldsymbol{\beta}_i^*\wt{B}^{\la i\ra} R_i G,\qquad\quad \widehat{d}_i \bar{g}_{ik}G R_i \wt{B}^{\la i\ra} \boldsymbol{\alpha}_i\boldsymbol{\beta}_i^*G. \label{022401}
\end{align}
Recalling the definition of $\Delta_{Z_j}(i,k)$ in~(\ref{022250}) and the matrix $D$ in~(\ref{022251}), we see that 
\begin{align}
 \sum_{j=1}^N d_j \Delta_{Z_j}(i,k)= & N\ntr \Delta_G(i,k) \text{tr} D -N\mathcal{A}_1\ntr \big(\Delta_G(i,k)D\big)\nonumber\\
 &\qquad-N\ntr \big(\Delta_G(i,k)AD\big) \ntr G+N\mathcal{A}_2\ntr \big(DG\big)\ntr \big((A-z)\Delta_G(i,k)\big)\nonumber\\
&\qquad+N\ntr \Delta_G(i,k)\Big( \mathcal{A}_3\ntr \big(DG\big)-\ntr \big(GAD\big)\Big)\nonumber\\
&\qquad+ N\ntr G \ntr \big(GD\big) \ntr\Big((A-z)^2\Delta_G(i,k)\Big). \label{022402}
\end{align}
Then, according to~(\ref{022401}) and~(\ref{022402}), we see that $\sum_j d_j \Delta_{Z_j}(i,k)$ is the sum of the terms of the form
\begin{align}
N\widehat{d}_i \bar{g}_{ik} \ntr \big(QG\boldsymbol{\alpha}_i\boldsymbol{\beta}_i^*\wt{B}^{\la i\ra} R_i G\big)=\widehat{d}_i \bar{g}_{ik} \boldsymbol{\beta}_i^*\wt{B}^{\la i\ra} R_i GQG\boldsymbol{\alpha}_i,\nonumber\\
N\widehat{d}_i \bar{g}_{ik}\ntr \big(QG R_i \wt{B}^{\la i\ra} \boldsymbol{\alpha}_i\boldsymbol{\beta}_i^*G\big)=\widehat{d}_i \bar{g}_{ik}\boldsymbol{\beta}_i^*GQG R_i \wt{B}^{\la i\ra} \boldsymbol{\alpha}_i, \label{022460}
\end{align}
for some random variables $\widehat{d}_i$, with $|\widehat{d}_i|\prec 1$, for all $i\in\llbracket1, N\rrbracket$, and some $i$-independent diagonal matrix~$Q$ with $\|Q\| \prec 1$, which can be $A$, $D$, $A-z$, $(A-z)^2$ or the product of some of them.

With the above facts, we can prove Lemmas~\ref{lem.022201},~\ref{lem.022405} and~\ref{lem.022407} in the sequel.
\begin{proof}[Proof of Lemma~\ref{lem.022201}] Using the fact that $\Delta_G(i,k)$ is a sum of the terms of the form in~(\ref{022401}), we see that the left side of~(\ref{022410}) is the sum of the terms of the form
\begin{align*}
&\frac{1}{N}\widehat{d}_i\sum_{k:k\neq i}\bar{g}_{ik} \mathbf{e}_k^* \wt{B}^{\la i\ra}G\boldsymbol{\alpha}_i\boldsymbol{\beta}_i^*\wt{B}^{\la i\ra} R_i G\mathbf{e}_i= \frac{1}{N}\widehat{d}_i\mathring{\mathbf{g}}_i^* \wt{B}^{\la i\ra}G\boldsymbol{\alpha}_i\boldsymbol{\beta}_i^*\wt{B}^{\la i\ra} R_i G\mathbf{e}_i=O_\prec(\frac{1}{N}),\nonumber\\
&\frac{1}{N}\widehat{d}_i\sum_{k:k\neq i}\bar{g}_{ik} \mathbf{e}_k^* \wt{B}^{\la i\ra}GR_i\wt{B}^{\la i\ra}\boldsymbol{\alpha}_i\boldsymbol{\beta}_i^*G\mathbf{e}_i=\frac{1}{N}\widehat{d}_i\mathring{ \mathbf{g}}_i^* \wt{B}^{\la i\ra}GR_i\wt{B}^{\la i\ra}\boldsymbol{\alpha}_i\boldsymbol{\beta}_i^*G\mathbf{e}_i=O_\prec(\frac{1}{N}),
\end{align*}
where we used $\mathring{\mathbf{g}}_i=\mathbf{g}_i-g_{ii}\mathbf{e}_i$,  the identities in~(\ref{030430}) and  the bound in~(\ref{030321}). Hence, we conclude the proof of Lemma~\ref{lem.022201}.
\end{proof}

\begin{proof}[Proof of Lemma~\ref{lem.022405}] Using that $ \sum_j d_j \Delta_{Z_j}(i,k)$ is a sum of such terms as in~(\ref{022460}) and $\sum_{k:k\neq i}\bar{g}_{ik}\mathbf{e}_k^*=\mathring{\mathbf{g}}_i^*$, we see that the left side of~(\ref{022461}) is the sum of the terms of the form
\begin{align*}
&\frac{1}{N^3}\sum_{i=1}^N\widehat{d}_i \mathring{\mathbf{g}}_i^*Q_iG\mathbf{e}_i \boldsymbol{\beta}_i^*\wt{B}^{\la i\ra} R_i GQG\boldsymbol{\alpha}_i=O_\prec(\Psi^4),\nonumber\\
&\frac{1}{N^3}\sum_{i=1}^N\widehat{d}_i \mathring{\mathbf{g}}_i^*Q_iG\mathbf{e}_i \boldsymbol{\beta}_i^*GQG R_i \wt{B}^{\la i\ra} \boldsymbol{\alpha}_i=O_\prec(\Psi^4),
\end{align*}
where we used $\mathring{\mathbf{g}}_i=\mathbf{g}_i-g_{ii}\mathbf{e}_i$,~(\ref{030430}) and~(\ref{030321}). This completes the proof of Lemma~\ref{lem.022405}.
\end{proof}

\begin{proof}[Proof of Lemma~\ref{lem.022407}] Using that $\Delta_G(i,k)$ is the sum of such terms as in~(\ref{022401}) and $\sum_{k:k\neq i}\bar{g}_{ik}\mathbf{e}_k^*=\mathring{\mathbf{g}}_i^*$, we see that the left side of~(\ref{022470})  is the sum of the terms of the form
\begin{align*}
&\frac{1}{N^2} \sum_{i=1}^N \widehat{d}_i \mathring{\mathbf{g}}_i^* Q_i\mathbf{g}_i \mathbf{e}_i^*G\boldsymbol{\alpha}_i\boldsymbol{\beta}_i^*\wt{B}^{\la i\ra} R_iG\mathbf{e}_i=O(\frac{1}{N}),\nonumber\\
&\frac{1}{N^2} \sum_{i=1}^N \widehat{d}_i \mathring{\mathbf{g}}_i^* Q_i\mathbf{g}_i \mathbf{e}_i^*GR_i\wt{B}^{\la i\ra}\boldsymbol{\alpha}_i\boldsymbol{\beta}_i^*G\mathbf{e}_i=O(\frac{1}{N}),
\end{align*}
where we used the fact $\mathring{\mathbf{g}}_i=\mathbf{g}_i-g_{ii}\mathbf{e}_i$,~(\ref{030430}) and~(\ref{030321}). 
Hence, we conclude the proof of Lemma~\ref{lem.022407}.
\end{proof}

\section{}\label{the Appendix C}

In this appendix, we discuss the case when both $\mu_\alpha$ and $\mu_\beta$  (\cf~(\ref{le assumptions convergence empirical measures})) are convex combinations of two point masses. Without loss of generality (up to shifting and scaling), we may assume that~$\mu_\alpha$ and~$\mu_\beta$ have the form
\begin{align}
&\mu_\alpha=\xi\delta_1+(1-\xi)\delta_0,\qquad\quad \mu_\beta=\zeta\delta_{\theta}+(1-\zeta)\delta_0,\label{081210}
\end{align}
with real parameters $\xi,\zeta$ and $\theta$ satisfying
\begin{align*}
 \theta\neq 0,\qquad\quad \xi,\zeta\in \Big(0,\frac{1}{2}\Big], \qquad \quad\xi\leq \zeta,\qquad \quad (\theta, \xi,\zeta)\neq \Big(-1,\frac12,\frac12\Big). 
\end{align*}
Recall the domains $\mathcal{S}_{\mathcal{I}}(a,b)$ in~\eqref{le domain S}. For given (small) $\varsigma,\gamma>0$, we set
\begin{align}
\mathcal{S}_\mathcal{I}^{\varsigma}(a,b)&\deq\bigg\{z\in \mathcal{S}_{\mathcal{I}}(a,b): \varsigma|z-1|\geq\max\Big\{\sqrt{{\rm{d_L}}(\mu_A,\mu_\alpha)},\sqrt{{\rm{d_L}}(\mu_B,\mu_\beta)}\Big\} \bigg\}\\
\widetilde{\mathcal{S}}_\mathcal{I}^{\varsigma}(a,b)&\deq\mathcal{S}_\mathcal{I}^{\varsigma}(a,b) \cap\bigg\{z\in\mathbb{C}: |z-1|\geq \frac{N^{\gamma}}{(N\eta)^{\frac14}}\bigg\}.\label{081501}
\end{align}
The following theorem presents the local law under the setting~(\ref{081210}).
\begin{thm}[Local law in the two point masses case] \label{thm. two point masses} Let $\mu_\alpha,\mu_\beta$ be as in~(\ref{081210}), with fixed $\xi,\zeta$ and~$\theta$.
Assume that the sequence of matrices $A$ and $B$ satisfy~(\ref{le bounded A and B}). Fix any compact nonempty interval $\mathcal{I}\subset \mathcal{B}_{\mu_\alpha\boxplus \mu_\beta}$. Then there is a constant $b>0$ such that if
 \begin{align}
  \dL(\mu_A,\mu_\alpha)+\dL(\mu_B,\mu_\beta)\le b,
 \end{align}
holds, then the following statements hold:

 \begin{itemize}[noitemsep,topsep=0pt,partopsep=0pt,parsep=0pt]
\item[$(i)$] If $\mu_\alpha\neq \mu_\beta$, then
\begin{align}
\Big|\frac{1}{N}\sum_{i=1}^N d_i \Big(G_{ii}(z)-\frac{1}{a_i-\omega_B(z)}\Big)\Big|\prec \Psi^2 \label{040801}
\end{align}
holds uniformly for all $z\in\mathcal{S}_{\mathcal{I}}(0,1)$.
Consequently,
\begin{align}
\sup_{\mathcal{I}'\subset\mathcal{I}}\Big|\mu_H(\mathcal{I}')-\mu_{A}\boxplus\mu_{B}(\mathcal{I}')\Big|\prec \frac{1}{N}, \label{040802}
\end{align}
where the supremum is over all subintervals of $\mathcal{I}$.
\item[$(ii)$] If $\mu_\alpha=\mu_\beta$, then, for sufficiently small $\varsigma>0$,
\begin{align}
\Big|\frac{1}{N}\sum_{i=1}^N d_i \Big(G_{ii}(z)-\frac{1}{a_i-\omega_B(z)}\Big)\Big|\prec \frac{\Psi^2}{|z-1|^2} \label{040803}
\end{align}
holds uniformly for all $z\in\wt{\mathcal{S}}^\varsigma_{\mathcal{I}}(0,1)$. Thus, for any nonempty compact interval $\widetilde{\mathcal{I}}\subset \mathcal{I}\setminus\{1\}$,
\begin{align}
\sup_{\mathcal{I}'\subset{\widetilde{\mathcal{I}}}}\Big|\mu_H(\mathcal{I}')-\mu_{A}\boxplus\mu_{B}(\mathcal{I}')\Big|\prec \frac{1}{N}, \label{040804}
\end{align}
where the supremum is over all subintervals of $\widetilde{\mathcal{I}}$.
\end{itemize}		
\end{thm}

Notice that the result  deviates from the general case from only if $\mu_\alpha=\mu_\beta$
due to an instability at $z=1$ in the free convolution $\mu_\alpha\boxplus \mu_\alpha$.

\begin{rem} For $\mu_\alpha,\mu_\beta$ given in~(\ref{081210}), the regular bulk $\mathcal{B}_{\mu_\alpha\boxplus\mu_\beta}$ can be written down explicitly, in terms of $\xi,\zeta$ and $\theta$, see (B.2) and (B.3) in~\cite{BES15b} for more detail.
\end{rem}

\begin{proof}
For ~(\ref{040801}) and~(\ref{040803}), analogously to~(\ref{fluctuation averaging}),  one needs to exploit the fluctuation average of the  $G_{ii}$'s, namely, that the fluctuation of the (weighted) average of  $G_{ii}$'s is typically as small as the square of the fluctuation of $G_{ii}$'s. Note that the estimate of the individual $G_{ii}$'s of the two point masses case  has been obtained in Proposition B.1 of~\cite{BES15b}. Since the proofs of~(\ref{040801}) and~(\ref{040803}) are nearly the same as~(\ref{fluctuation averaging}), given Proposition B.1 of~\cite{BES15b}, we omit the details. Then the convergence
rates~(\ref{040802}) and~(\ref{040804}) follow from~(\ref{040801}) and~(\ref{040803}), respectively,  via a routine application of the Helffer-Sj\"{o}strand functional calculus; see~\eg Section~7.1 of~\cite{EKYY13}. This completes the proof.
\end{proof}

\end{document}